\documentclass[reqno,11pt]{amsart}
\usepackage{a4wide}
\usepackage{amsmath,bm}
\usepackage{mathrsfs}
\usepackage{amsfonts}
\usepackage{indentfirst}
\usepackage{stmaryrd}
\usepackage{leftidx}
\usepackage{accents}
\usepackage{mathtools}
\usepackage{amssymb}
\usepackage{amsthm}
\usepackage[all]{xy}
\usepackage{color}
\usepackage{verbatim}
\usepackage{cite}
\usepackage{enumitem}
\usepackage[latin1]{inputenc}
\usepackage{color}
\usepackage{xr-hyper}
\usepackage[
pdftex,
bookmarks=false,
colorlinks=true,
debug=true,
pdfnewwindow=true]{hyperref}

\usepackage[T1]{fontenc}
\usepackage{float}
\usepackage{subfigure}
\usepackage{micro type}
\allowdisplaybreaks[4]
\numberwithin{equation}{section}

\newtheorem{Thm}{Theorem}[section]
\newtheorem{Rk}[Thm]{Remark}
\newtheorem{Def}[Thm]{Definition}
\newtheorem{Prop}[Thm]{Proposition}
\newtheorem{Lem}[Thm]{Lemma}
\newtheorem{Cor}[Thm]{Corollary}

\def\QA{U_{v}(\widehat{\mathfrak{sl}}(2|1))}
\def\QAP{U_{v}^{>}(\widehat{\mathfrak{sl}}(2|1))}
\def\QAG{U_{v}(\widehat{\mathfrak{sl}}(n|m))}

\def\dto{U_{v}^{>}(\widehat{\mathfrak{D}}(2,1;\theta))}

\newcommand\nc{\newcommand}
\nc{\unl}{\underline}
\nc{\ol}{\overline}
\nc{\on}{\operatorname}

\nc{\BA}{{\mathbb{A}}}
\nc{\BC}{{\mathbb{C}}}
\nc{\BD}{{\mathbb{D}}}
\nc{\BF}{{\mathbb{F}}}
\nc{\BG}{{\mathbb{G}}}
\nc{\BM}{{\mathbb{M}}}
\nc{\BN}{{\mathbb{N}}}
\nc{\BO}{{\mathbb{O}}}
\nc{\BQ}{{\mathbb{Q}}}
\nc{\BP}{{\mathbb{P}}}
\nc{\BR}{{\mathbb{R}}}
\nc{\BZ}{{\mathbb{Z}}}
\nc{\BS}{{\mathbb{S}}}
\nc{\BK}{{\mathbb{K}}}

\nc{\CA}{{\mathcal{A}}} \nc{\CB}{{\mathcal{B}}} \nc{\CalC}{{\mathcal
C}} \nc{\CalD}{{\mathcal D}} \nc{\CE}{{\mathcal{E}}}
\nc{\CF}{{\mathcal{F}}} \nc{\CG}{{\mathcal{G}}}
\nc{\CH}{{\mathcal{H}}} \nc{\CI}{{\mathcal{I}}}
\nc{\CK}{{\mathcal{K}}} \nc{\CL}{{\mathcal{L}}}
\nc{\CM}{{\mathcal{M}}} \nc{\CN}{{\mathcal{N}}}
\nc{\CO}{{\mathcal{O}}} \nc{\CP}{{\mathcal{P}}}
\nc{\CQ}{{\mathcal{Q}}} \nc{\CR}{{\mathcal{R}}}
\nc{\CS}{{\mathcal{S}}} \nc{\CT}{{\mathcal{T}}}
\nc{\CU}{{\mathcal{U}}} \nc{\CV}{{\mathcal{V}}}
\nc{\CW}{{\mathcal{W}}} \nc{\CX}{{\mathcal{X}}}
\nc{\CY}{{\mathcal{Y}}} \nc{\CZ}{{\mathcal{Z}}}

\nc{\fa}{{\mathfrak{a}}}
\nc{\fb}{{\mathfrak{b}}}
\nc{\fg}{{\mathfrak{g}}}
\nc{\fgl}{{\mathfrak{gl}}}
\nc{\fh}{{\mathfrak{h}}}
\nc{\fj}{{\mathfrak{j}}}
\nc{\fl}{{\mathfrak{l}}}
\nc{\fm}{{\mathfrak{m}}}
\nc{\fn}{{\mathfrak{n}}}
\nc{\fu}{{\mathfrak{u}}}
\nc{\fp}{{\mathfrak{p}}}
\nc{\frr}{{\mathfrak{r}}}
\nc{\fs}{{\mathfrak{s}}}
\nc{\ft}{{\mathfrak{t}}}
\nc{\fw}{{\mathfrak{w}}}
\nc{\fz}{{\mathfrak{z}}}

\nc{\fA}{{\mathfrak{A}}}
\nc{\fB}{{\mathfrak{B}}}
\nc{\fD}{{\mathfrak{D}}}
\nc{\fE}{{\mathfrak{E}}}
\nc{\fF}{{\mathfrak{F}}}
\nc{\fG}{{\mathfrak{G}}}
\nc{\fI}{{\mathfrak{I}}}
\nc{\fJ}{{\mathfrak{J}}}
\nc{\fK}{{\mathfrak{K}}}
\nc{\fL}{{\mathfrak{L}}}
\nc{\fM}{{\mathfrak{M}}}
\nc{\fN}{{\mathfrak{N}}}
\nc{\frP}{{\mathfrak{P}}}
\nc{\fQ}{{\mathfrak Q}}
\nc{\fR}{{\mathfrak R}}
\nc{\fS}{{\mathfrak S}}
\nc{\fT}{{\mathfrak{T}}}
\nc{\fU}{{\mathfrak{U}}}
\nc{\fW}{{\mathfrak{W}}}
\nc{\fY}{{\mathfrak{Y}}}
\nc{\fZ}{{\mathfrak{Z}}}

\nc{\ba}{{\mathbf{a}}}
\nc{\bb}{{\mathbf{b}}}
\nc{\bc}{{\mathbf{c}}}
\nc{\bd}{{\mathbf{d}}}
\nc{\be}{{\mathbf{e}}}
\nc{\bi}{{\mathbf{i}}}
\nc{\bj}{{\mathbf{j}}}
\nc{\bn}{{\mathbf{n}}}
\nc{\bp}{{\mathbf{p}}}
\nc{\bq}{{\mathbf{q}}}
\nc{\bu}{{\mathbf{u}}}
\nc{\bv}{{\mathbf{v}}}
\nc{\bw}{{\mathbf{w}}}
\nc{\bx}{{\mathbf{x}}}
\nc{\by}{{\mathbf{y}}}
\nc{\bz}{{\mathbf{z}}}

\nc{\bA}{{\mathbf{A}}}
\nc{\bB}{{\mathbf{B}}}
\nc{\bC}{{\mathbf{C}}}
\nc{\bD}{{\mathbf{D}}}
\nc{\bE}{{\mathbf{E}}}
\nc{\bI}{{\mathbf{I}}}
\nc{\bK}{{\mathbf{K}}}
\nc{\bH}{{\mathbf{H}}}
\nc{\bM}{{\mathbf{M}}}
\nc{\bN}{{\mathbf{N}}}
\nc{\bO}{{\mathbf{O}}}
\nc{\bQ}{{\mathbf Q}}
\nc{\bS}{{\mathbf{S}}}
\nc{\bT}{{\mathbf{T}}}
\nc{\bV}{{\mathbf{V}}}
\nc{\bW}{{\mathbf{W}}}
\nc{\bX}{{\mathbf{X}}}
\nc{\bP}{{\mathbf{P}}}
\nc{\bY}{{\mathbf{Y}}}
\nc{\bZ}{{\mathbf{Z}}}

\nc{\sA}{{\mathsf{A}}}
\nc{\sB}{{\mathsf{B}}}
\nc{\sC}{{\mathsf{C}}}
\nc{\sD}{{\mathsf{D}}}
\nc{\sF}{{\mathsf{F}}}
\nc{\sK}{{\mathsf{K}}}
\nc{\sM}{{\mathsf{M}}}
\nc{\sO}{{\mathsf{O}}}
\nc{\sQ}{{\mathsf{Q}}}
\nc{\sP}{{\mathsf{P}}}
\nc{\sT}{{\mathsf{T}}}
\nc{\sV}{{\mathsf{V}}}
\nc{\sW}{{\mathsf{W}}}
\nc{\sX}{{\mathsf{X}}}
\nc{\sZ}{{\mathsf{Z}}}
\nc{\sU}{{\mathsf{U}}}
\nc{\sS}{{\mathsf{S}}}

\nc{\sfb}{{\mathsf{b}}}
\nc{\sfc}{{\mathsf{c}}}
\nc{\sd}{{\mathsf{d}}}
\nc{\sg}{{\mathsf{g}}}
\nc{\sk}{{\mathsf{k}}}
\nc{\sfl}{{\mathsf{l}}}
\nc{\sfp}{{\mathsf{p}}}
\nc{\sr}{{\mathsf{r}}}
\nc{\st}{{\mathsf{t}}}
\nc{\sfu}{{\mathsf{u}}}
\nc{\sw}{{\mathsf{w}}}
\nc{\sz}{{\mathsf{z}}}
\nc{\sx}{{\mathsf{x}}}
\nc{\se}{{\mathsf{e}}}
\nc{\sfv}{{\mathsf{v}}}

\nc{\bLambda}{{\boldsymbol{\Lambda}}}
\nc{\vv}{{\boldsymbol{v}}}
\nc{\Fl}{{{\mathcal F}\ell}}
\nc{\Gr}{{\on{Gr}}}
\nc{\CHH}{{\CH\!\!\CH}}
\nc{\lambdavee}{{\lambda^{\!\scriptscriptstyle\vee}}}
\nc{\alphavee}{\alpha^{\!\scriptscriptstyle\vee}}
\nc{\rhovee}{{\rho^{\!\scriptscriptstyle\vee}}}

\nc{\oQM}{\vphantom{j^{X^2}}\smash{\overset{\circ}{\vphantom{\vstretch{0.7}{A}}\smash{\QM}}}}
\nc{\oZ}{{}^\dagger\!\vphantom{j^{X^2}}\smash{\overset{\circ}{\vphantom{\vstretch{0.7}{A}}\smash{Z}}}}
\nc{\odZ}{{}^\dagger\!\vphantom{j^{X^2}}\smash{\overset{\circ}{\vphantom{\vstretch{0.7}{A}}\smash{\mathfrak Z}}}^{c',c}}
\nc{\bdZ}{{}^\dagger\!\vphantom{j^{X^2}}\smash{\overset{\bullet}{\vphantom{\vstretch{0.7}{A}}\smash{\mathfrak Z}}}^{c',c}}
\nc{\oS}{\vphantom{j^{X^2}}\smash{\overset{\circ}{\vphantom{\vstretch{0.7}{A}}\smash{S}}}}
\nc{\buM}{\vphantom{j^{X^2}}\smash{\overset{\bullet}{\vphantom{\vstretch{0.7}{A}}\smash{M}}}}
\nc{\dW}{{}^\dagger\ol\CW{}}
\nc{\hW}{{}^\dagger\hat\CW{}}
\nc{\wW}{{}^\dagger\widetilde\CW{}}
\nc{\dZ}{{}^\dagger\!\fZ^{c',c}}
\nc{\dZc}{{}^\dagger\!\fZ^{c,c}}
\nc{\tZ}{{}^\dagger\!\tilde{Z}{}}
\nc{\hZ}{{}^\dagger\!\hat{Z}{}}

\nc{\ssl}{\mathfrak{sl}} \nc{\gl}{\mathfrak{gl}}
\nc{\wt}{\widetilde} \nc{\Sym}{\mathrm{Sym}} \nc{\Res}{\mathrm{Res}}
\nc{\sE}{{\mathsf{E}}} \nc{\bs}{{\mathbf{s}}}
\nc{\trig}{\mathrm{trig}} \nc{\rat}{\mathrm{rat}}
\nc{\sign}{\mathrm{sign}} \nc{\sL}{{\mathsf{L}}}
\nc{\fv}{{\mathfrak{v}}} \nc{\ad}{\mathrm{ad}}
\nc{\spsi}{{\mathsf{\psi}}} \nc{\sh}{{\mathsf{h}}}
\nc{\rtt}{\mathrm{rtt}} \nc{\qdet}{\mathrm{qdet}} \nc{\pt}{{\operatorname{pt}}}
\nc{\M}{\mathrm{M}} \nc{\Ker}{\mathrm{Ker}} \nc{\ssc}{\mathrm{sc}}
\nc{\loc}{\mathrm{loc}} \nc{\fra}{\mathrm{frac}}
\nc{\ddj}{\mathrm{DJ}} \nc{\End}{\mathrm{End}} \nc{\ev}{\mathrm{ev}}
\nc{\GL}{\mathrm{GL}}

\begin{document}

\title[]{Shuffle algebra realization of\\ quantum affine  superalgebra $U_{v}(\widehat{\mathfrak{D}}(2,1;\theta))$}

\author[]{Boris Feigin}
\address[]{National Research University Higher School of Economics, Russian Federation, International Laboratory of Representation Theory, Mathematical Physics, Russia, Moscow, 101000, Myasnitskaya ul., 20\\Landau Institute for Theoretical Physics, Russia, Chernogolovka, 142432, pr.Akademika
Semenova, 1a.
}
\email[]{borfeigin@gmail.com}
\author[]{Yue Hu}
\address[]{Center for Advanced Studies,
Skolkovo Institute of Science and Technology,
1 Nobel Street, Moscow, 143026, Russia}
\email[]{ldkhtys@gmail.com}

\date{}

\begin{abstract}
Inspired by \cite{Tsy18}, we give shuffle algebra realization of positive part of quantum affine superalgebra $U_{v}(\widehat{\mathfrak{D}}(2,1;\theta))$ associated to any simple root systems. We also determine the shuffle algebra associated to $\widehat{\mathfrak{sl}}(2|1)$ with odd root system when $v$ is a primitive root of unity of even order, generalizing results in \cite{FJMMT03}.
\end{abstract}

\maketitle

\section{Introduction}
Shuffle algebras are certain (skew)symmetric Laurent polynomials with prescribed poles satisfying the so called wheel conditions, and endowed with an associative algebra structure by shuffle product. These algebras are first studied by the first author and Odesskii in \cite{FO98}. They are interesting because they are expected to give realizations of quantum affine and quantum toroidal (super)algebras. The known examples are for type $A$ cases. In \cite{SV13}, Schiffmann and Vasserot constructed an isomorphism between the shuffle algebra associated to $\widetilde{A}_{1}$ and the positive part of the elliptic Hall algebra, or equivalently, the positive part of quantum toroidal $\ddot{U}_{v_{1},v_{2}}(\mathfrak{gl}(1))$ algebra, see also \cite{Neg14} for more details. In \cite{Neg13}, Negut generalized this result to higher rank cases, and proved that the shuffle algebra associated to $\widetilde{A}_{n}$ is isomorphic to the positive part of quantum toroidal $\ddot{U}_{v_{1},v_{2}}(\mathfrak{gl}(n))$ algebra for $n\geqslant 2$. For other types of finite Dynkin diagrams, a conjectural shuffle algebra realizations had been given, see \cite{Enr00},\cite{Enr03}.

It is interesting to even further consider the Dynkin diagrams associated to Kac-Moody superalgebras. In \cite{Tsy18}, Tsymbaliuk gave the shuffle algebra realization for quantum affine superalgebra $U_{v}(\widehat{\mathfrak{sl}}(m|n))$ with distinguished simple root system. His results suggest that in the super case, we should consider the skew-symmetric rational functions instead of symmetric ones corresponding to the odd simple roots. Note that the Kac-Moody superalgebras admit nonisomorphic simple root systems, and they give different positive parts. Recently in \cite{Tsy19}, Tsymbaliuk generalized results in \cite{Tsy18} to all simple root systems associated to $\mathfrak{sl}(m|n)$ and gave shuffle algebra realizations of the corresponding quantum affine superalgebras, making the picture for $A(m|n)$ case complete.

 In this paper, we consider the exceptional Lie superalgebra $\mathfrak{D}(2,1;\theta)$ and give shuffle algebra realization of positive part of quantum affine superalgebra $U_{v}(\widehat{\mathfrak{D}}(2,1;\theta))$ associated to any simple root systems, see the proof of Theorem \ref{d21} and Theorem \ref{d212}. Note that this shuffle algebra realization can be easily extended from the positive part to the whole algebra, see \cite {Neg13} for more details. The problem of giving shuffle algebra realization for quantum toroidal $\ddot{U}_{v_{1},v_{2}}(\mathfrak{D}(2,1;\theta))$ algebra has been posed in \cite{FJM18} to study the quantization of $\widehat{\mathfrak{sl}}_{2}$ coset vertex operator algebra, and our motivations start from there.

We give an outline of our proofs and state the meaning of our results. First we define the shuffle algebra $\Omega$ associated to $\mathfrak{D}(2,1;\theta)$, by finding certain wheel conditions that are used to replace the role of quantum Serre relations in the quantum affine algebra $\dto$. Then there is a natural morphism $\varphi$ from $\dto$ to $\Omega$ in Drinfeld realization. To prove the surjectivity of $\varphi$, following ideas in \cite{Tsy18}, we construct certain ordered monomials of quantum affine root vectors as PBW type elements in $\dto$ and show their images under $\varphi$ constitute a basis for $\Omega$. The difficulty is that the standard specialization map used in \cite{Tsy18}, which is one main tool when studying shuffle algebras in type $A$ cases, behaved badly in our case. We overcome this by defining a more complicated specialization map that is compatible with the wheel conditions in our setting. We believe that our results shine a light on giving an answer to the conjectural shuffle algebra realizations for any finite Dynkin diagrams.

To prove the injectivity of $\varphi$, we choose a different method from Tsymbaliuk's. Similar to the type $A$ case considered in \cite{HRZ08}, we show that in our case the ordered monomials of quantum affine root vectors also span the whole algebra, thus the linearly independence of their images in shuffle algebras would give us the injectivity of this morphism. While Tsymbaliuk's idea is based on the existence of compatible nondegenerate pairings on both sides, see \cite[Proposition 3.4]{Tsy18} and \cite{Neg13} for more details.

As a byproduct, we construct PBW type bases for $\dto$ in the Drinfeld realization, which shows the benefits of shuffle algebra realization of quantum affine algebras. Note that the PBW bases for quantum affine algebras had been established a long time ago in the standard Drinfeld-Jimbo presentation, there seems to be missing in literatures a clear proof of PBW property for them in the Drinfeld realization, for more details on this see the introduction in \cite{Tsy18}. This proof of PBW property for quantum affine algebras in Drinfeld realization by comparing them to the corresponding shuffle algebra is a natural generalization of the usual proof of PBW bases theorem for quantum enveloping algebras, which is by comparing them to the skew symmetric algebras.

In generators and relations, $\mathfrak{D}(2,1;\theta)$ with distinguished simple root system is constructed by gluing three $\mathfrak{sl}(2|1)$ algebras using Serre relations, thus we first give the shuffle algebra realization of $\QAP$ in odd simple root system. Moreover, we also determine this shuffle algebra when $v$ is a primitive root of unity. When $v$ is generic, shuffle algebras are generated by degree one elements. However, when $v$ is a primitive root of unity, the degree one elements only generate a subalgebra, and we need more wheel conditions to determine it. For example, the positive part of $U_{v}(\widehat{\mathfrak{sl}}(2))$ is isomorphic to the symmetric Laurent polynomials with shuffle product, and under this isomorphism the PBW bases correspond to Hall-Littlewood Laurent polynomials. When $v$ is a primitive root of unity, the corresponding shuffle algebra generated by degree one elements consists of symmetric Laurent polynomials spanned by ``admissible'' Hall-Littlewood Laurent polynomials. It is proved that this subspace is determined by certain wheel condition, see \cite[Proposition 3.5]{FJMMT03}. For $\QA$, its positive part is isomorphic to doubly skew-symmetric Laurent polynomials with prescribed poles. We show when $v$ is a primitive root of unity of even order, the corresponding shuffle algebra generated by degree one elements is also determined by certain wheel condition, see the proof of Theorem \ref{wheelcondition}.

When we initiate this work, the paper \cite{Tsy19} had not came out and the shuffle algebra realization of $\QAG$ with non-distinguished simple root system was still unknown, so we give a detailed proof of shuffle algebra realization for $\QA$ with odd root system when $v$ is generic. We choose to preserve this part because our arguments differ from Tsymbaliuk's arguments in some parts and it is also needed for other parts of this paper. Also it can be served as an introduction to shuffle algebras by studying an example with details.

The paper is organized as follows. In Section \ref{U2}, we define the shuffle algebra $\Lambda$ associated to odd simple root system of $\mathfrak{sl}(2|1)$ and prove the isomorphism $\varphi\colon \QAP\xrightarrow{\sim} \Lambda$. When $v$ is a primitive root of unity of even order, the shuffle product can be still defined, we consider the subalgebra $\Lambda^{\zeta}$ generated by degree one elements, and prove that it is isomorphic to the subalgebra $\Lambda^{w}$ defined by certain wheel condition. In Section \ref{Ud}, we give shuffle algebra realization of $\dto$ associated to all simple root systems and prove their PBW property. 

\subsection*{Acknowledgements}We are grateful to Alexander Tsymbaliuk, whose generous help and advice was crucial in the process of our work. We also would like to thank Michael Finkelberg for useful discussions and suggestions.

\section{shuffle realization of $U_{v}^{>}(\widehat{\mathfrak{sl}}(2|1))$}\label{U2}

\subsection{$U_{v}^{>}(\widehat{\mathfrak{sl}}(2|1))$ and a spanning set}

Consider the free $\mathbb{Z}$-module $\oplus_{i=1}^{3}\epsilon_{i}$ with bilinear form $(\epsilon_{i},\epsilon_{j})=(-1)^{\delta_{i=3}}\delta_{ij}$. Instead of the distinguished simple root system $\{\epsilon_{1}-\epsilon_{2},\epsilon_{2}-\epsilon_{3}\}$, we choose the simple roots to be $\{\alpha_{1}=\epsilon_{1}-\epsilon_{3},\alpha_{2}=\epsilon_{3}-\epsilon_{2}\}$, which both are odd roots. The positive roots are $\Psi^{+}=\{\alpha_{1},\alpha_{2},\gamma=\alpha_{1}+\alpha_{2}\}$. The Cartan matrix is $\bigl(\begin{smallmatrix} 0 & 1 \\ 1 & 0 \end{smallmatrix}\bigr)$. 
 Following \cite[Theorem 8.5.1]{Yam99}, in the Drinfeld realization, $\QAP$ is the quantum superalgebra over $\mathbb{C}(v)$ with generators $\{p_{i}, q_{i}, i\in \mathbb{Z}\}$ and relations
\begin{equation}
\begin{aligned}
p_{i}p_{j}+p_{j}p_{i}&=0,\\
q_{i}q_{j}+q_{j}q_{i}&=0,\\
p_{i+1}q_{j}+vq_{j}p_{i+1}&=-vp_{i}q_{j+1}-q_{j+1}p_{i},\label{eq:relation21}
\end{aligned}
\end{equation}
here the parity of generators are given by $p(p_{i})=p(q_{j})=1$ and we denote by $[x,y]_{u}:=xy-(-1)^{|x||y|}uyx$ the super bracket. We will simply write $[x,y]$ for $[x,y]_{1}$. The following formulae can be directly checked from the above defining relations \eqref{eq:relation21}.

\begin{Lem}\label{comm}
\begin{enumerate}[leftmargin=*]
\item $q_{s}p_{k}=v[p_{k},q_{s}]_{v^{-1}}-vp_{k}q_{s}$.
\item $[p_{k},q_{s}]_{v^{-1}}+v[p_{k+1},q_{s-1}]_{v^{-1}}=(v-v^{-1})p_{k+1}q_{s-1}$.
\item $q_{s}[p_{k},q_{s}]_{v^{-1}}=v[p_{k},q_{s}]_{v^{-1}}q_{s}$.
\item $[p_{k},q_{s}]_{v^{-1}}p_{k}=vp_{k}[p_{k},q_{s}]_{v^{-1}}$.
\end{enumerate}
\end{Lem}

We will also use the following formulae for super bracket, see \cite[6.9]{Yam99}.

\begin{Lem}[6.9,\cite{Yam99}]\label{bra}
Let $U$ be a superalgebra over $\mathbb{C}(v)$. For any $X,Y,Z\in U$ and $a,b,c\in \mathbb{C}(v)$, we have
\begin{equation}
\begin{aligned}
&[[X,Y]_{a},Z]_{b}=[X,[Y,Z]_{c}]_{abc^{-1}}+(-1)^{p(Y)p(Z)}c[[X,Z]_{bc^{-1}},Y]_{ac^{-1}},\\
&[X,[Y,Z]_{a}]_{b}=[[X,Y]_{c},Z]_{abc^{-1}}+(-1)^{p(X)p(Y)}c[Y,[X,Z]_{bc^{-1}}]_{ac^{-1}}.
\end{aligned}
\end{equation}
\end{Lem}

Following \cite[Subsection 2.2]{Tsy18}, let $r_{i}=[p_{i},q_{0}]_{v^{-1}}$, see also \cite[Definition 3.9]{HRZ08} and \cite[Definition 3.11]{Zha14}. Then $\{p_{i},q_{j},r_{k}\}_{i,j,k\in\mathbb{Z}}$ are quantum affine root vectors corresponding to positive roots. Let $H$ be the set of functions $h\colon \Psi^{+}\times \mathbb{Z}\rightarrow \mathbb{N}$ with finite support and such that $h(\alpha_{i},k)\leqslant 1$. Now for any $\beta\in\Psi^{+}$, since $h$ has finite support, the set of integers $i\in\mathbb{Z}$ with multiplicity $h(\beta,i)\neq 0$ form a partition $\lambda_{\beta}=(\lambda_{\beta,1}\leqslant\cdots\leqslant\lambda_{\beta,k_{\beta}})$. We can define an ordered monomial $E_{h}$ as follows
\begin{equation}
\prod_{1\leqslant i\leqslant k_{\alpha_{1}}}p_{\lambda_{\alpha_{1},i}}\prod_{1\leqslant i\leqslant k_{\gamma}}r_{\lambda_{\gamma,i}}\prod_{1\leqslant i\leqslant k_{\alpha_{2}}}q_{\lambda_{\alpha_{2},i}}.
\end{equation}
Let $U'\subset\QAP$ be the spanning set of these $E_{h}$ over $\mathbb{C}(v)$.

\begin{Prop}\label{commutation}
For any $i,j,k,s\in\mathbb{Z}$, the elements $[p_{i},q_{j}]_{v^{-1}},[p_{i},r_{k}]_{v^{-1}},[q_{j},r_{k}]_{v^{-1}},[r_{k},r_{s}]_{v^{2}}$ are all belonging to $U'$.
\end{Prop}

\begin{proof}
We can assume $i,j\geqslant k\geqslant s\geqslant 0$, other cases are similar. First by Lemma \ref{comm} (2) we have $[p_{i},q_{j}]_{v^{-1}}=(-v)^{j}r_{i+j}+(v-v^{-1})\sum_{k=1}^{j}(-v)^{k-1}p_{i+k}q_{j-k}$. Hence $[p_{i},q_{j}]_{v^{-1}}\in U'$. Next by Lemma \ref{comm} (2),(4) we have $[p_{i},r_{k}]_{v^{-1}}=(v^{-1}-v)\sum_{l=1}^{i-k}(-v)^{l-1}p_{k+l}[p_{i},q_{-l}]_{v^{-1}}$. Hence $[p_{i},r_{k}]_{v^{-1}}\in U'$, and similarly by Lemma \ref{comm} (2),(3) we get $[q_{j},r_{k}]_{v^{-1}}\in U'$. Finally, we deal with $[r_{k},r_{s}]_{v^{2}}$. By Lemma \ref{comm} and Lemma \ref{bra} we have 
\begin{equation}
\begin{aligned}
{[r_{k},r_{s}]_{v^{2}}}&=[[p_{k},q_{0}]_{v^{-1}},[p_{s},q_{0}]_{v^{-1}}]_{v^{2}}\\
&=[p_{k},[q_{0},[p_{s},q_{0}]_{v^{-1}}]_{v}]+v[[p_{k},[p_{s},q_{0}]_{v^{-1}}]_{v},q_{0}]_{v^{-2}}\\
&=-[[p_{k},[p_{s+1},q_{-1}]_{v}]_{v},q_{0}]_{v^{-2}}\\
&=v[[p_{s+1},r_{k-1}]_{v},q_{0}]_{v^{-2}}\\
&=[r_{s+1},r_{k-1}]_{v^{2}}+v[p_{s+1},[r_{k-1},q_{0}]_{v^{-1}}]\\
&=[r_{s+1},r_{k-1}]_{v^{2}}\\
&=-v^{2}[r_{k-1},r_{s+1}]_{v^{2}}+(1-v^{4})r_{s+1}r_{k-1}.
\label{commrel}
\end{aligned}
\end{equation}
Thus if $k-s>1$ then $[r_{k},r_{s}]_{v^{2}}\in U'$ if and only if $[r_{k-1},r_{s+1}]_{v^{2}}\in U'$, so we only need to prove $[r_{k},r_{k}]_{v^{2}}\in U'$, which is obvious, and $[r_{k},r_{k-1}]_{v^{2}}\in U'$. In \eqref{commrel} let $s=k-1$ we have $[r_{k},r_{k-1}]_{v^{2}}=-[[p_{k},[p_{k},q_{-1}]_{v}]_{v},q_{0}]_{v^{-2}}=0$, this completes our proof.
\end{proof}

Now we have our main theorem of this subsection, that $U'$ actually equals to $\QAP$.
\begin{Thm}\label{pbw2}
The set of ordered monomials $\{E_{h}\}_{h\in H}$ is a spanning set for $\QAP$.
\end{Thm}

\begin{proof}
For any nonzero word $w=e_{1}\cdot e_{2}\cdots e_{n}\in \QAP$, in which each $e_{i}$ is $p_{j}$ or $q_{k}$, we call $n$ to be the length of $w$, and denote it by $l(w)=n$. For any element $x$ in $\QAP$, $x$ is a finite sum of words and define its length $l(x)$ to be the maximal length of its words. Using the commutation relations given in Proposition \ref{commutation}, we will prove that any word, hence any element $x$ in $\QAP$, is a finite linear combination of the above ordered monomials $E_{h}$, and we will prove it by induction on the length of elements. In the following we will omit the unnecessary coefficients in $\mathbb{C}(v)$ in the equations, and use the symbol $\approx$ to denote an equation without considering the coefficients. For example if $A=c_{1}(v)B+c_{2}(v)C$, then we have $A\approx B+C$. Suppose for any $w$ such that $l(w)\leqslant n$, the above statement is true. Now take any $w$ such that $l(w)=n+1$. If $w=p_{i_{0}}\cdot w'$, then by induction we are done. If $w=q_{i_{0}}\cdot w'$, then by induction we can assume $w'$ is an ordered monomial $E_{h}$. If $w'=w''\cdot q_{k_{0}}$, then by induction we are done. If $w'=p_{k_{0}}\cdot w''$, then $q_{i_{0}}\cdot w'=q_{i_{0}}p_{k_{0}}\cdot w''\approx r_{k_{0}+i_{0}}\cdot w''+\sum p_{i}q_{j}\cdot w''$. If $w''=\prod r_{j}$, then by induction we are done. If $w''=p_{j_{0}}\cdot w'''$, then $r_{k_{0}+i_{0}}p_{j_{0}}\cdot w'''\approx \sum p_{i}r_{j}\cdot w'''+\sum p_{i}p_{j}q_{k}\cdot w'''$, and by induction we are done. Finally, let $w'=\prod_{j=1}^{n} r_{t_{j}}=r_{t_{1}}\cdot w''$, then $w=q_{i_{0}}\cdot w'=q_{i_{0}}r_{t_{1}}\cdot w''\approx\sum r_{i}q_{j}\cdot w''+\sum p_{i}q_{j}q_{k}\cdot w''$, using induction on the element $q_{j}\cdot w''$ we get back to the cases considered above and thus we are done.
\end{proof}

\subsection{Shuffle algebra $\Lambda$}\label{sl21}

Let $\Lambda=\bigoplus_{n,m\in \mathbb{N}}\Lambda_{n,m}$ be graded vector spaces over $\mathbb{C}(v)$, where $\Lambda_{n,m}$ consists of rational functions $F$ in the variables $\{x_{1},\dots,x_{n},y_{1},\dots,y_{m}\}$ and satisfies the following conditions:
\begin{enumerate}[leftmargin=*]
\item $F$ is skew-symmetric with respect to $\{x_{i}\}_{1\leqslant i\leqslant n}$ and $\{y_{j}\}_{1\leqslant j\leqslant m}$.
\item $F=\frac{f}{\prod_{1\leqslant i\leqslant n,1\leqslant j\leqslant m}(x_{i}-y_{j})}$, where $f\in \mathbb{C}(v)[x_{i}^{\pm 1},y_{j}^{\pm 1}]$ is a Laurent polynomial.
 \end{enumerate}

Denote by $\mathfrak{S}_{n}$ the symmetric group of order $n$. For $F\in \Lambda_{k_{1},l_{1}}, G\in \Lambda_{k_{2},l_{2}}$, we define the shuffle product $F\star G\in \Lambda_{k_{1}+k_{2},l_{1}+l_{2}}$ as

\begin{equation}
\begin{aligned}
&F\star G=\text{ASym}_{\mathfrak{S}_{k_{1}+k_{2}}\times \mathfrak{S}_{l_{1}+l_{2}}}\Big(F(\{x_{i},y_{j}\}_{1\leqslant i\leqslant k_{1}}^{1\leqslant j\leqslant l_{1}})\cdot G(\{x_{i},y_{j}\}_{k_{1}<i\leqslant k_{1}+k_{2}}^{l_{1}<j\leqslant l_{1}+l_{2}})\\
&\cdot\prod_{1\leqslant i\leqslant k_{1}}^{ l_{1}+1\leqslant j\leqslant l_{1}+l_{2}}\frac{x_{i}+v^{-1}y_{j}}{x_{i}-y_{j}}\prod_{k_{1}+1\leqslant i\leqslant k_{1}+k_{2}}^{ 1\leqslant j\leqslant l_{1}}\frac{y_{j}+v^{-1}x_{i}}{y_{j}-x_{i}}\Big),\label{eq:sp}
 \end{aligned}
\end{equation}
where $\text{ASym}_{\mathfrak{S}_{n}\times\mathfrak{S}_{m}}$ means skew-symmetrization with respect to $\{x_{i}\}_{1\leqslant i\leqslant n}$ and $\{y_{j}\}_{1\leqslant j\leqslant m}$, that is for any rational function $f(x_{1},\dots,x_{n},y_{1},\dots,y_{m})$,
\begin{equation}
\text{ASym}_{\mathfrak{S}_{n}\times\mathfrak{S}_{m}}(f)=\frac{1}{n!m!}\sum_{\sigma\in\mathfrak{S}_{n},\tau\in\mathfrak{S}_{m}}\text{sign}(\sigma)\text{sign}(\tau)f(x_{\sigma(1)},\dots,x_{\sigma(n)},y_{\tau(1)},\dots,y_{\tau(m)}).
\end{equation}

Standardly, we have

\begin{Prop}
Under the shuffle product $\star$, $\Lambda$ is an associative $\mathbb{C}(v)$-algebra.
\end{Prop}
\begin{proof}
See the proof of \cite[Lemma 2.3]{FHHSY09}.
\end{proof}

\subsection{Isomorphism between $\QAP$ and $\Lambda$}\label{isosl21}There is a natural $C(v)$-algebra morphism $\varphi$ from $\QAP$ to $\Lambda$. Our aim is to prove $\varphi$ is actually an isomorphism.
\begin{Prop}
$p_{i}\mapsto x^{i}, q_{j}\mapsto y^{j}$ induces a $\mathbb{C}(v)$-algebra morphism $\varphi\colon \QAP\rightarrow \Lambda$.
\end{Prop}

\begin{proof}
This is straightforward to check. For example let us verify the last defining relation in \eqref{eq:relation21}, since $\varphi(p_{r})=x^{r},\varphi(q_{s})=y^{j}$ we have
$$x^{r+1}\star y^{s}+vy^{s}\star x^{r+1}=\frac{(v^{-1}-v)x^{r+1}y^{s+1}}{x-y}=-vx^{r}\star y^{s+1}-y^{s+1}\star x^{r}.$$
In particular, under $\varphi$ the quantum affine root vector $r_{k}$ has the following explicit from 
$$\varphi(r_{k})=\frac{(1-v^{-2})x^{k+1}}{x-y}.$$
\end{proof}

\begin{Lem}\label{elementary}
Denote $\underbrace{\frac{1}{x-y}\star\cdots\star\frac{1}{x-y}}_{n}\star\  y^{0}\star y^{1}\star\cdots\star y^{k-1}$ by $P_{n,k}$, then $P_{n,k}$ is equal to 
$$c\cdot v^{-kn}\prod_{i=1}^{n}\frac{1-v^{-2i}}{1-v^{-2}}\cdot\frac{\prod_{1\leqslant i<j\leqslant n}(x_{i}-x_{j})\prod_{1\leqslant t<l\leqslant n+k}(y_{t}-y_{l})}{\prod_{1\leqslant i\leqslant n}^{1\leqslant j\leqslant n+k}(x_{i}-y_{j})},$$
where $c\neq 0\in\mathbb{C}$ depends on $n$ and $k$. 

Similar result holds also for $Q_{n,k}=x^{0}\star x^{1}\star\cdots\star x^{k-1}\star\underbrace{\frac{1}{x-y}\star\cdots\star\frac{1}{x-y}}_{n}$.
\end{Lem}

\begin{proof}
By definition, under skew-symmetrization, $\prod_{i<j}(x_{i}-x_{j})\prod_{k<l}(y_{k}-y_{l})$ is a factor. And by comparing degrees between the two sides, we know it is the only factor. So we only need to prove $P_{n,k}\neq 0$, i.e., $\text{ASym}(\prod_{i<j}(x_{i}+v^{-1}y_{j})\prod_{i<j}(y_{i}+v^{-1}x_{j})y_{n+2}\cdots y_{n+k}^{k-1})\neq 0$.

First, when $k=0$, let $\Gamma_{n}=\prod_{1\leqslant i<j\leqslant n}(x_{i}+v^{-1}y_{j})(y_{i}+v^{-1}x_{j})$, then the coefficient of $x_{1}^{n-1}y_{1}^{n-1}$ in $\text{ASym}(\Gamma_{n})$ is equal to $\frac{1-v^{-2n}}{n^{2}(1-v^{-2})}\text{ASym}(\Gamma_{n-1}(x_{2},\dots,x_{n},y_{2},\dots,y_{n}))$, now by induction on $n$ we are done. When $k>0$, assume the statement is true for $k-1$, and let $\Delta_{n,k}=v^{-kn}\prod_{i=1}^{n}\frac{1-v^{-2i}}{1-v^{-2}}\prod_{1\leqslant i<j\leqslant n}(x_{i}-x_{j})\prod_{1\leqslant i<j\leqslant n+k}(y_{i}-y_{j})$. Then by the associativity of shuffle product we have $P_{n,k}=P_{n,k-1}\star y^{k-1}=$
\begin{equation}
\frac{c\cdot\Delta_{n,k-1}}{\prod_{1\leqslant i\leqslant n}^{1\leqslant j\leqslant n+k-1}(x_{i}-y_{j})}\star y^{k-1}=\frac{c\cdot\text{ASym}(\Delta_{n,k-1}\prod_{1\leqslant i\leqslant n}(x_{i}+v^{-1}y_{n+k})y_{n+k}^{k-1})}{\prod_{1\leqslant i\leqslant n}^{1\leqslant j\leqslant n+k}(x_{i}-y_{j})},
\end{equation}
then the coefficient of $y_{n+k}^{n+k-1}$ in the numerator part of $P_{n,k}$ is $\frac{c}{n+k}v^{-n}\Delta_{n,k-1}$, hence by induction on $k$ we are done.

For $Q_{n,k}$ the proof is the same.
\end{proof}

\begin{Prop}\label{ideal}
$\Lambda$ is generated by $\{x^{i},y^{j}\}_{i,j\in \mathbb{Z}}$, i.e., $\varphi$ is surjective.
\end{Prop}

\begin{proof}
We need to prove for each monomial $f(x_{1},\dots,x_{n},y_{1},\dots,y_{m})=x_{1}^{a_{1}}\cdots x_{n}^{a_{n}}y_{1}^{b_{1}}\cdots y_{m}^{b_{m}}$, $a_{i},b_{j}\in \mathbb{Z}$, $\frac{\prod_{i<j}(x_{i}-x_{j})\prod_{k<l}(y_{k}-y_{l})\text{Sym}_{\mathfrak{S}_{n}\times\mathfrak{S}_{m}}(f)}{\prod(x_{i}-y_{j})}\in \Lambda_{n,m}$, where $\text{Sym}_{\mathfrak{S}_{n}\times\mathfrak{S}_{m}}$ denotes symmetrization with respect to $\mathfrak{S}_{n}\times \mathfrak{S}_{m}$, that is
\begin{equation}
\text{Sym}_{\mathfrak{S}_{n}\times\mathfrak{S}_{m}}(f)=\frac{1}{n!m!}\sum_{\sigma\times\tau\in\mathfrak{S}_{n}\times\mathfrak{S}_{m}}f(x_{\sigma(1)},\dots,x_{\sigma(n)},y_{\tau(1)},\dots,y_{\tau(m)}).
\end{equation}
 We can assume $n\leqslant m$ and $m=n+k$. For each $\sigma=\pi\times\tau\in \mathfrak{S}_{n}\times \mathfrak{S}_{m}$, let $F_{\sigma}\in \Lambda_{n,m}$ be
$$F_{\sigma}=\frac{x^{a_{\pi(1)}}y^{b_{\tau(1)}}}{x-y}\star\cdots\star\frac{x^{a_{\pi(n)}}y^{b_{\tau(n)}}}{x-y}\star y^{b_{\tau(n+1)}}\star y^{b_{\tau(n+2)}+1}\star\cdots\star y^{b_{\tau(n+k)}+k-1}.$$
Then $\frac{1}{n!m!}\sum_{\sigma\in \mathfrak{S}_{n}\times \mathfrak{S}_{m}}F_{\sigma}=$
$$\frac{\text{Sym}_{\mathfrak{S}_{n}\times\mathfrak{S}_{m}}(f)\cdot\text{ASym}(\prod_{i<j}(x_{i}+v^{-1}y_{j})\prod_{i<j}(y_{i}+v^{-1}x_{j})y_{n+2}\cdots y_{n+k}^{k-1})}{\prod(x_{i}-y_{j})}.$$
By Lemma \ref{elementary} we get the desired element.
\end{proof}



Now by Theorem \ref{pbw2} and the surjectivity of $\varphi$, we know $\{\varphi(E_{h})\}_{h\in H}$  span the shuffle algebra $\Lambda$. Thus proving their linearly independence would gives us the linearly independence of $\{E_{h}\}_{h\in H}$ and that $\varphi$ is actually an isomorphism. The following proposition follows from the technique of specialization introduced in \cite{Tsy18}, we will treat this particular case as an example to explain it. Note that this technique of specialization has been frequently used in the studies of shuffle algebras, see \cite{FHHSY09,Neg13}.

\begin{Prop}\label{spe21}
$\{\varphi(E_{h})\}_{h\in H}$ are linearly independent in $\Lambda$.
\end{Prop}

\begin{proof}
 For any $h\in H$, denote its degree by $\text{deg}(h)=\underline{d}=(d_{1},d_{2},d_{3})\in \mathbb{N}^{3}$ such that $d_{1}=\sum_{k\in \mathbb{Z}} h(\alpha_{1},k)$, $d_{3}=\sum_{k\in \mathbb{Z}} h(\alpha_{2},k)$, $d_{2}=\sum_{k\in \mathbb{Z}} h(\gamma,k)$. If $\text{deg}(h)=(d_{1},d_{2},d_{3})$, we call $\text{gr}(h)=(d_{1}+d_{2},d_{2}+d_{3})\in\mathbb{N}^{2}$ its grading and we have $\varphi(E_{h})\in \Lambda_{\text{gr}(h)}$.  We hope to prove in each graded part $\Lambda_{n,m}$, the elements $\{\varphi(E_{h}),\text{gr}(h)=(n,m)\}$ are linearly independent. For any $h,h'\in H$ such that $\text{gr}(h)=\text{gr}(h')=(n,m)$, we say $\deg{h'}<\deg{h}$ if  $d'_{1}<d_{1}$, and it induces a complete order on the set of degree of functions that has grading $(n,m)$ and we list them as $D_{n,m}=\{\underline{d}_{1}<\cdots <\underline{d}_{l}\}$.  Now for $F\in \Lambda_{n,m}$ and $\underline{d}\in D_{n,m}$, we define the specialization map $\phi_{\underline{d}}\colon \Lambda_{n,m}\rightarrow V_{\underline{d}}=\mathbb{C}(v)[z_{1,1}^{\pm 1},\dots,z_{1,d_{1}}^{\pm 1},z_{2,1}^{\pm 1},\dots,z_{2,d_{2}}^{\pm 1},z_{3,1}^{\pm 1},\dots,z_{3,d_{3}}^{\pm 1}]$. Let $f$ be the numerator part of $F$, then $\phi_{\underline{d}}(F)$ is the corresponding Laurent polynomial by specializing the variables in $f$ as follows:
\begin{equation}
\begin{aligned}
&x_{i}\mapsto z_{1,i},1\leqslant i\leqslant d_{1},\\
&x_{i}\mapsto z_{2,i-d_{1}},d_{1}+1\leqslant i\leqslant n,\\
&y_{j}\mapsto -vz_{2,j},1\leqslant j\leqslant d_{2},\\
&y_{j}\mapsto -vz_{3,j-d_{2}},d_{2}+1\leqslant j\leqslant m.
\end{aligned}
\end{equation}
Since $f$ is skew-symmetric with respect to $\{x_{i}\}_{1\leqslant i\leqslant n}$ and $\{y_{j}\}_{1\leqslant j\leqslant m}$, $\phi_{\underline{d}}(F)$ is skew-symmetric with respect to $\{z_{i,1},\dots,z_{i,d_{i}}\}$ for any $1\leqslant i\leqslant 3$. Since $\underline{d}_{l}=(n,0,m)$, we know $\phi_{\underline{d}_{l}}(F)$ is equal to $f(z_{1,1},\dots,z_{1,n},-vz_{3,1},\dots,-vz_{3,m})$, thus $\phi_{\underline{d}_{l}}(F)\neq 0$ if and only if $F\neq 0$. Also for $2\leqslant k\leqslant l$ if $\underline{d}_{k}=(d_{1},d_{2},d_{3})$, then $\underline{d}_{k-1}=(d_{1}-1,d_{2}+1,d_{3}-1)$. Let $g\in V_{\underline{d}_{k}}$ and define the specialization map $\rho_{k}\colon V_{\underline{d}_{k}}\rightarrow V_{\underline{d}_{k-1}}$ such that $\rho_{k}(g)$ is the corresponding Laurent polynomial by specializing the variables in $g$ as follows:
\begin{equation}
\begin{aligned}
&z_{1,i}\mapsto z_{1,i},1\leqslant i\leqslant d_{1}-1,\\
&z_{1,d_{1}},z_{3,1}\mapsto z_{2,d_{2}+1},\\
&z_{2,i}\mapsto z_{2,i},1\leqslant i\leqslant d_{2},\\
&z_{3,i}\mapsto z_{3,i-1},2\leqslant i\leqslant d_{3}.
\end{aligned}
\end{equation}
Then we have $\phi_{\underline{d}_{k-1}}=\rho_{k}\circ\phi_{\underline{d}_{k}}$. In addition we let $\phi_{\underline{d}_{0}}$ be the zero map. Hence we have a filtration on $\Lambda_{n,m}$:
$$\Lambda_{n,m}=\text{Ker}(\phi_{\underline{d}_{0}})\supset\cdots\supset\text{Ker}(\phi_{\underline{d}_{l-1}})\supset\text{Ker}(\phi_{\underline{d}_{l}})=0.$$
Now for any $h\in H$ and the associated partitions $\lambda_{\alpha_{1}}=(a_{1}<\cdots<a_{d_{1}})$, $\lambda_{\gamma}=(b_{1}\leqslant\cdots\leqslant b_{d_{2}})$, $ \lambda_{\alpha_{2}}=(c_{1}<\cdots<c_{d_{3}})$, $\varphi(E_{h})$ equals to 

\begin{equation}
\begin{aligned}
 &\text{Asym}_{\mathfrak{S}_{n}\times \mathfrak{S}_{m}}\Big(x_{1}^{a_{1}}\cdots x_{d_{1}}^{a_{d_{1}}}\cdot x_{d_{1}+1}^{b_{1}+1}\cdots x_{n}^{b_{d_{2}}+1}\cdot y_{d_{2}+1}^{c_{1}}\cdots y_{m}^{c_{d_{3}}}\prod^{1\leqslant j\leqslant m}_{i\leqslant d_{1}}(x_{i}+v^{-1}y_{j})\\
 &\prod^{j>d_{2}}_{i>d_{1}}(x_{i}+v^{-1}y_{j})\prod_{1\leqslant i<j\leqslant d_{2}}(x_{i+d_{1}}+v^{-1}y_{j})(y_{i}+v^{-1}x_{d_{1}+j})\Big)/\prod(x_{i}-y_{j}),
 \end{aligned}
\end{equation}
it is a sum of terms corresponding to elements of $\mathfrak{S}_{n}\times \mathfrak{S}_{m}$. First let us compute $\phi_{\underline{d}}(\varphi(E_{h}))$. In this case the terms which do not specialize to zero are corresponding to those $\sigma\times \tau\in \mathfrak{S}_{n}\times \mathfrak{S}_{m}$ such that $\sigma=\sigma_{1}\times \pi,\tau=\pi\times \tau_{1}$, where $\sigma_{1}\in \mathfrak{S}_{d_{1}},\pi\in \mathfrak{S}_{d_{2}},\tau_{1}\in \mathfrak{S}_{d_{3}}$. Hence we have $\phi_{\underline{d}}(\varphi(E_{h}))=$
\begin{equation}
\begin{aligned}
&Z\cdot \text{Asym}_{\mathfrak{S}_{d_{1}}}(z_{1,1}^{a_{1}}\cdots z_{1,d_{1}}^{a_{d_{1}}})\cdot \text{Asym}_{\mathfrak{S}_{d_{3}}}( z_{3,1}^{c_{1}}\cdots z_{3,d_{3}}^{c_{d_{3}}})\\
&\cdot \text{Sym}_{\mathfrak{S}_{d_{2}}}(z_{2,1}^{b_{1}+1}\cdots z_{2,d_{2}}^{b_{d_{2}}+1}\prod_{1\leqslant i<j\leqslant d_{2}}\frac{z_{2,i}-v^{2}z_{2,j}}{z_{2,i}-z_{2,j}}),
\label{speresult}
\end{aligned}
\end{equation}
where $Z=c(v)\cdot\prod_{1\leqslant i\leqslant d_{1},1\leqslant j\leqslant d_{2}}^{1\leqslant k\leqslant d_{3}}(z_{1,i}-z_{2,j})(z_{1,i}-z_{3,k})(z_{2,j}-z_{3,k})\prod_{1\leqslant i<j\leqslant d_{2}}(z_{2,i}-z_{2,j})^{2}$ is a common factor for all $h$, and $c(v)=c\cdot v^{l}$ for some $c\neq 0\in\mathbb{C},l\in\mathbb{Z}$. Let $A_{n}$ be the space of skew-symmetric Laurent polynomials of $n$ variables over $\mathbb{C}(v)$. From the above formula we know if we take all $h\in H$ such that $\deg{h}=\underline{d}$, then under $\phi_{\underline{d}}$ the images of $\{\varphi(E_{h})\}$ constitute a basis for $Z\cdot A_{d_{1}}\cdot A_{d_{3}}\cdot S_{d_{2}}$, hence $\{\varphi(E_{h})\}_{\deg{h}=\underline{d}}$ is linearly independent set. It is also clear from the above explicit formula for $\varphi_{\underline{d}}(E_{h})$ that for $\underline{d}'<\deg{h}$, we have $\phi_{\underline{d}'}(\varphi(E_{h}))=0$. Hence for any $\deg{h}=\underline{d}_{k}$ we have $\varphi(E_{h})\subset \text{Ker}(\phi_{\underline{d}_{k-1}})-\text{Ker}(\phi_{\underline{d}_{k}})$ and thus collect them all the set $\{\varphi(E_{h})\}_{h\in D_{n,m}}$ is linearly independent.
\end{proof}

\begin{Rk}\label{remark}\rm
Using the above specialization map $\phi_{\underline{d}}$, we can also give a proof of surjectivity of $\varphi$, following \cite[Lemma 3.19]{Tsy18}. We only need to prove for any $F\in \Lambda_{n,m}$, $F$ is equal to a linear combination of some $\varphi(E_{h})$ such that $h\in H$ and $\text{gr}(h)=(n,m)$. Without loss of generality, we assume $n\leqslant m$. If $\phi_{\underline{d}'}(F)=0$ for any $\underline{d}'<\underline{d}$, then we see $\phi_{\underline{d}}(F)$ has $\prod_{1\leqslant i\leqslant d_{1}}^{1\leqslant k\leqslant d_{3}}(z_{1,i}-z_{3,k})$ as a factor. To see this recall we have $\phi_{\underline{d}_{k-1}}=\rho_{k}\circ\phi_{\underline{d}_{k}}$, hence $\phi_{\underline{d}'}(F)=0$ for any $\underline{d'}<\underline{d}$ shows that $\phi_{\underline{d}}(F)$ has the factor $z_{1,d_{1}}-z_{3,1}$, and since $\phi_{\underline{d}}(F)$ is skew-symmetric with respect to $\{z_{i,j}\}_{1\leqslant i\leqslant 3}^{1\leqslant j\leqslant d_{i}}$, by taking symmetrization we get the desired factor. Moreover since $F$ is skew-symmetric, we see $\phi_{\underline{d}}(F)$ is exactly some linear combination of elements $\phi_{\underline{d}}(\varphi(E_{h}))$. Note that if we choose $\underline{d}=\underline{d}_{1}$, that is if $d_{1}=0$, then $\phi_{\underline{d}_{1}}(F)$ automatically has the factor $\prod_{1\leqslant i<j\leqslant d_{2}}(z_{2,i}-z_{2,j})^{2}\prod_{1\leqslant j\leqslant d_{2}}^{1\leqslant k\leqslant d_{3}}(z_{2,j}-z_{3,k})$. Thus there is some $G_{1}$ which is some linear combinations of $\varphi(E_{h})$ such that $\deg{h}=\underline{d}_{1}$ and $\phi_{\underline{d}_{1}}(F)=\phi_{\underline{d}_{1}}(G_{1})$, that is $\phi_{\underline{d}_{1}}(F-G_{1})=0$. Hence we have $G_{2}$ that is some linear combination of $\varphi(E_{h})$ such that $\deg{h}=\underline{d}_{2}$ and $\phi_{\underline{d}_{2}}(F-G_{1})=\phi_{\underline{d}_{2}}(G_{2})$. Repeat this procedure, we have $G_{1},\dots,G_{l}$ which are all linear combinations of $\varphi(E_{h})$ such that $\phi_{\underline{d}_{l}}(F)=\phi_{\underline{d}_{l}}(G_{1}+\cdots+G_{l})$. Since $\text{Ker}(\phi_{\underline{d}_{l}})=0$, we have $F=G_{1}+\cdots+G_{l}$ and is some linear combination of $\varphi(E_{h})$. Similar arguments can also prove the surjectivity  for $\dto$, once we define the appropriate specialization map, see the proof of Theorem \ref{d21} and Theorem \ref{d212}.
\end{Rk}

\begin{Cor}
$\{E_{h}\}_{h\in H}$ are PBW type bases for $\QAP$.
\end{Cor}

\begin{proof}
By Theorem \ref{pbw2} we know they span the whole $\QAP$, and by Proposition \ref{spe21} we know they are linearly independent.
\end{proof}

\begin{Thm}
$\varphi\colon\QAP\rightarrow \Lambda$ is an isomorphism.
\end{Thm}

\begin{proof}
We prove the surjectivity of $\varphi$ in Proposition \ref{ideal}, now since $\{E_{h}\}_{h\in H}$ are bases and \{$\varphi(E_{h})\}$ are linearly independent, $\varphi$ is also injective.
\end{proof}

\subsection{When $v$ is a primitive root of unity}


In this subsection we will let $v\in\mathbb{C}$ be a complex number and study a root of unity version of $\Lambda$. We actually give a generalization and a new proof of results in \cite{FJMMT03}. 

Let $S=\oplus_{k\in\mathbb{N}}S_{k}=\oplus_{k\in\mathbb{N}}\mathbb{C}[x_{1}^{\pm 1},\dots,x_{k}^{\pm 1}]^{\mathfrak{S}_{k}}$ be the graded vector space of symmetric Laurent polynomials over $\mathbb{C}$. Similar to the above when $v$ is a parameter, for $F\in S_{k},G\in S_{l}$, we can define the shuffle product $F\star G\in S_{k+l}$ as
\begin{equation}
F\star G=\text{Sym}_{\mathfrak{S}_{k+l}}\Big(F(\{x_{i}\}_{1\leqslant i\leqslant k}) G(\{x_{j}\}_{k<j\leqslant k+l})\prod_{i\leqslant k}^{j>k}\frac{x_{i}-vx_{j}}{x_{i}-x_{j}}\Big).
\end{equation}
Using this shuffle product $S$ becomes an associative algebra over $\mathbb{C}$. For any partition $\lambda$ of length $n$, the Hall-Littlewood polynomial with $n$ variables is defined as

\begin{equation}
P_{\lambda}(x_{1},\dots,x_{n};v)=\text{Sym}_{\mathfrak{S}_{n}}(x_{1}^{\lambda_{1}}\cdots x_{n}^{\lambda_{n}}\prod_{1\leqslant i<j\leqslant n}\frac{x_{i}-vx_{j}}{x_{i}-x_{j}}).
\end{equation}
It is well known that when $v$ is generic, the Hall-Littlewood polynomials form a $\mathbb{C}$ basis of $S$. By definition of above shuffle product, we have $P_{\lambda}(x_{1},\dots,x_{n};v)=x^{\lambda_{1}}\star\cdots\star x^{\lambda_{n}}$, hence when $v$ is generic, $S$ is generated by $S_{1}$. 

Now let $v$ be a primitive root of unity of order $t\in\mathbb{N}$, and consider the subalgebra $S'$ generated by $S_{1}$ over $\mathbb{C}$. In this case $S'$ is not equal to $S$, and we hope to determine it by certain conditions and construct a basis for it. For any $F\in S$, $F$ is said to be satisfying the wheel condition if $F(x,vx,\dots,v^{t-1}x,x_{t+1},\dots)=0$ for any $x\in \mathbb{C}$. Denote by $S^{w}\subset S$ the subspace consisting of elements satisfying the wheel condition. For a partition $\lambda$ of length $n$, denote by $m_{i}(\lambda)$ the number of $i$ appearing in $\lambda$, we say $\lambda$ is admissible if $m_{i}(\lambda)\leqslant t-1$ for any $i\in \mathbb{Z}$. One main result of \cite{FJMMT03} is

\begin{Thm}[Proposition 3.5,\cite{FJMMT03}]\label{boris}
When $v$ is a primitive root of unity of order $t$, the Hall-Littlewood polynomials $P_{\lambda}$ in which $\lambda$ is admissible form a basis of $S'$ over $\mathbb{C}$. Moreover, $S'=S^{w}$.
\end{Thm}

Now let $\overline{\Lambda}=\bigoplus_{n,m\in \mathbb{N}}\overline{\Lambda}_{n,m}$ be graded vector spaces over $\mathbb{C}$, where $\overline{\Lambda}_{n,m}$ consists of rational functions $F$ in the variables $\{x_{1},\dots,x_{n},y_{1},\dots,y_{m}\}$ satisfying the same conditions in $\Lambda$. Let $v$ be a primitive root of unity of order $2t$, then \eqref{eq:sp} defines an associative algebra structure on $\overline{\Lambda}$ over $\mathbb{C}$. Let $\Lambda^{\zeta}$ be the subalgebra generated by $\overline{\Lambda}_{1}=\overline{\Lambda}_{1,0}\oplus\overline{\Lambda}_{0,1}$. As in the $\QAP$ case, for any $h\in H$, define $F_{h}\in \overline{\Lambda}$ as
\begin{equation}
x^{\lambda_{\alpha_{1},1}}\star\cdots \star x^{\lambda_{\alpha_{1},k_{\alpha_{1}}}}\star\frac{x^{\lambda_{\gamma,1}}}{x-y}\star\cdots\star \frac{x^{\lambda_{\gamma,k_{\gamma}}}}{x-y}\star y^{\lambda_{\alpha_{2},1}}\star\cdots \star y^{\lambda_{\alpha_{2},k_{\alpha_{2}}}}.
\end{equation}
We say $h\in H$ is admissible if $m_{i}(\lambda_{\gamma})\leqslant t-1$ for any $i\in\mathbb{Z}$ and we denote by $H^{a}$ the set of admissible functions in $H$. Then we have

\begin{Prop}
$\{F_{h}\}_{h\in H^{a}}$ form a $\mathbb{C}$-basis of $\Lambda^{\zeta}$.
\end{Prop}

\begin{proof}
From the proof of Lemma \ref{elementary}, we know for any $k\in \mathbb{Z}$
$$\underbrace{\frac{x^{k}}{x-y}\star \cdots\star \frac{x^{k}}{x-y}}_{m}=0$$ 
if and only if $m\geqslant t$. Hence $F_{h}=0$ if $h\notin H^{a}$. The specialization map $\phi_{\underline{d}}$ defined in Proposition \ref{spe21} still applies when $v$ is a complex number, thus by \eqref{speresult} we know $\{F_{h}\}_{h\in H^{a}}$ are linearly independent over $\mathbb{C}$. Since $\Lambda^{\zeta}$ is generated by $\{x^{i},y^{i}\}_{i\in\mathbb{Z}}$, and they satisfy the same relations \eqref{eq:relation21} as $\{p^{i},q^{i}\}_{i\in\mathbb{Z}}$, thus Theorem \ref{pbw2} shows $\{F_{h}\}_{h\in H^{a}}$ is also a spanning set for $\Lambda^{\zeta}$ over $\mathbb{C}$.
\end{proof}

Similar to \cite{FJMMT03}, we prove that $\Lambda^{\zeta}$ is also governed by certain wheel condition.

\begin{Def}
When $v$ is a primitive root of unity of order $2t$, $F\in \overline{\Lambda}$ is said to satisfy the wheel condition if 
\begin{equation}
F(x_{1},\dots,x_{n},y_{1},\dots,y_{m})=0 \text{ once } \frac{x_{1}}{y_{1}}=\frac{y_{1}}{x_{2}}=\frac{x_{2}}{y_{2}}=\cdots=\frac{x_{t}}{y_{t}}=\frac{y_{t}}{x_{1}}=-v^{-1}.\label{wheelunity}
\end{equation}
\end{Def}

We denote the subspace of elements in $\overline{\Lambda}$ satisfying the wheel condition by $\Lambda^{w}$.
\begin{Prop}\label{wheel}
$\Lambda^{w}$ is a subalgebra of $\overline{\Lambda}$ under shuffle product and $\Lambda^{\zeta}\subset\Lambda^{w}$.
\end{Prop}

\begin{proof}
Let $F,G\in \Lambda^{w}$, we shall prove that each term of \eqref{eq:sp} is zero under the specialization $\frac{x_{1}}{y_{1}}=\frac{y_{1}}{x_{2}}=\cdots=\frac{y_{t}}{x_{1}}=-v^{-1}$. Note that each term is corresponding to a permutation $\sigma\times \tau$. In the following for $\sigma^{-1}(t+1)$ we mean $\sigma^{-1}(1)$. For $1\leqslant i\leqslant k_{1}+k_{2}$, we denote $\text{sgn}(i)=1$ if $1\leqslant i\leqslant k_{1}$ and $\text{sgn}(i)=2$ otherwise. Define similarly for $1\leqslant j\leqslant l_{1}+l_{2}$. Now let $\sigma\in \mathfrak{S}_{k_{1}+k_{2}}, \tau\in \mathfrak{S}_{l_{1}+l_{2}}$, we see if there is some $1\leqslant i\leqslant t$ such that $\text{sgn}(\sigma^{-1}(i))=1,\text{sgn}(\tau^{-1}(i))=2$ or $\text{sgn}(\sigma^{-1}(i+1))=2,\text{sgn}(\tau^{-1}(i))=1$, then this term is specialized to zero. Otherwise it must happen that $\text{sgn}(\sigma^{-1}(i))=\text{sgn}(\tau^{-1}(i))$ for all $1\leqslant i\leqslant t$, then since $F,G\in \Lambda^{w}$ this term is also specialized to zero. Since for $f\in \overline{\Lambda}_{1}$ the wheel condition becomes nothing, we know $\overline{\Lambda}_{1}\subset \Lambda^{w}$. Since $\Lambda^{w}$ is an algebra, we know $\Lambda^{\zeta}\subset \Lambda^{w}$.
\end{proof}

Certainly, if $F\in \Lambda^{w}_{n,m}$ and $\text{min}\{n,m\}<t$, then $F\in \Lambda^{\zeta}$. Slightly further, we have

\begin{Prop}\label{toy}
If $F\in \Lambda^{w}_{t,t}$, then $F\in \Lambda^{\zeta}$.
\end{Prop}

\begin{proof}
If $t=1$ it is trivial, hence we assume $t\geqslant 2$. Since each $F\in\Lambda^{w}$ is of the form $\frac{f}{\prod(x_{i}-y_{j})}$, where $f=\prod(x_{i}-x_{j})\prod(y_{k}-y_{l})\cdot g$ and $g$ is a symmetric polynomial with respect to $\{x_{i}\}$ and $\{y_{j}\}$, we know $F$ satisfies the wheel condition if and only if $g$ satisfies the wheel condition. Let $\{\chi_{1},\dots,\chi_{t}\},\{\psi_{1},\dots,\psi_{t}\}$ be separately the elementary symmetric polynomials of $\{x_{1},\dots,x_{t}\}$ and $\{y_{1},\dots,y_{t}\}$, then $g=G(\chi_{1},\dots,\chi_{t},\psi_{1},\dots,\psi_{t})$ for some polynomial $G$. The wheel condition says $g(x,\dots,v^{2t-2}x,-vx,\dots,-v^{2t-1}x)=0$, it is equivalent to $G(0,\dots,0,(-1)^{t-1}x^{t},0,\dots,0,x^{t})=0$ for any $x\in \mathbb{C}$. Hence $g$ satisfies the wheel condition if and only if $g$ belongs to the ideal generated by $\{\chi_{1},\dots,\chi_{t-1},\psi_{1},\dots,\psi_{t-1},\chi_{t}+(-1)^{t}\psi_{t}\}$. Now it is easy to check that for $1\leqslant r<t$,
$$\underbrace{\frac{x}{x-y}\star \cdots\star \frac{x}{x-y}}_{r}\star\underbrace{\frac{1}{x-y}\star \cdots\star \frac{1}{x-y}}_{t-r}=\frac{c(v)\cdot\chi_{r}\cdot\Delta_{t,0}}{\prod(x_{i}-y_{j})},$$
$$\underbrace{\frac{1}{x-y}\star \cdots\star \frac{1}{x-y}}_{t-r}\star\underbrace{\frac{y}{x-y}\star \cdots\star \frac{y}{x-y}}_{r}=\frac{c(v)\cdot\psi_{r}\cdot\Delta_{t,0}}{\prod(x_{i}-y_{j})},$$
$$\underbrace{\frac{x+y}{x-y}\star \cdots\star \frac{x+y}{x-y}}_{t-2}\star\frac{x}{x-y}\star\frac{y}{x-y}=\frac{c(v)\cdot(\chi_{t}+(-1)^{t}\psi_{t}+L)\cdot\Delta_{t,0}}{\prod(x_{i}-y_{j})},$$
where $L$ belongs to the ideal generated by $\{\chi_{1},\dots,\chi_{t-1},\psi_{1},\dots,\psi_{t-1}\}$. Now by the proof of Proposition \ref{ideal}, we know if $F\in \Lambda^{\zeta}_{n,m}$, then for any symmetric Laurent polynomial $G\in S_{n,m}$ we have $G\cdot F\in \Lambda^{\zeta}$. Hence the above elements also generate an ideal and it equals to $\Lambda^{w}_{t,t}$.
\end{proof}

Viewing Proposition \ref{toy} as a toy model and starting point of induction, we can now prove the general case.

\begin{Thm}\label{wheelcondition}
For any $f\in \overline{\Lambda}$, $f\in\Lambda^{\zeta}$ if and only if $f$ satisfies the wheel condition \eqref{wheelunity}.
\end{Thm}

\begin{proof}
We will focus on the symmetric factor. For $k,l\geqslant 0$, let $F\in \overline{\Lambda}_{t+k,t+l}$, then the corresponding symmetric factor $g$ satisfies the wheel condition if and only if $g$ belongs to the ideal generated by $\{\chi_{k+1},\dots,\chi_{t+k-1},\psi_{l+1},\dots,\psi_{t+l-1},\chi_{t+k}\psi_{k}+(-1)^{t}\chi_{l}\psi_{t+l}\}$, here $\chi_{0}=\psi_{0}=1$. Now it is easy to check that $\chi_{t+i}$ is generated by shuffle product of $x$ and $\chi_{t+i-1}$, $\psi_{t+j}$ is generated by shuffle product of $\psi_{t+j-1}$ and $y$. By Proposition \ref{toy} and induction on $k,l$, we get $\Lambda^{\zeta}=\Lambda^{w}$.
\end{proof}

\begin{Rk}\rm
Our method of proving Theorem \ref{wheelcondition} actually gives a new proof of Theorem \ref{boris}.
\end{Rk}

\section{shuffle realization of $U_{v}^{>}(\widehat{\mathfrak{D}}(2,1;\theta))$}\label{Ud}

\subsection{Drinfeld realization and a spanning set} The exceptional Lie superalgebras $\mathfrak{D}(2,1;\theta)$ with $\theta\in\mathbb{C}$ and $\theta\neq 0,-1$ form a one-parameter family of superalgebras of rank $3$ and dimension $17$. There are four different simple root systems and corresponding Dynkin diagrams, in this subsection we choose the completely fermionic one. Namely the simple roots are $\{\alpha_{1},\alpha_{2},\alpha_{3}\}$ with parities $p(\alpha_{i})=1$ for $i=1,2,3$ and Cartan matrix $A=(a_{ij})_{1\leqslant i,j\leqslant 3}$ where
$$A=\begin{pmatrix}
0&1&\theta\\
1&0&-\theta-1\\
\theta&-\theta-1&0
\end{pmatrix}.$$

The positive roots are $\Psi^{+}=\{\alpha_{1}\prec\alpha_{1}+\alpha_{3}\prec\alpha_{1}+\alpha_{2}\prec\alpha_{1}+\alpha_{2}+\alpha_{3}\prec\alpha_{2}\prec\alpha_{2}+\alpha_{3}\prec\alpha_{3}\}$ with a fixed ordering. We also fix an ordering on $\Psi^{+}\times\mathbb{Z}$ as follows:
\begin{equation}
(\beta_{1},k_{1})\prec(\beta_{2},k_{2})\Leftrightarrow \beta_{1}\prec\beta_{2}\text{ or }\beta_{1}=\beta_{2},k_{1}\leqslant k_{2}.
\end{equation}
 The odd positive roots are $\Psi_{1}^{+}=\{\alpha_{1},\alpha_{2},\alpha_{3},\alpha_{123}=\alpha_{1}+\alpha_{2}+\alpha_{3}\}$, the even positive roots are $\Psi_{0}^{+}=\{\alpha_{12}=\alpha_{1}+\alpha_{2},\alpha_{23}=\alpha_{2}+\alpha_{3},\alpha_{13}=\alpha_{1}+\alpha_{3}\}$. The quantum affine superalgebra $U_{v}(\widehat{\mathfrak{D}}(2,1;\theta))$ has been studied in \cite{HSTY08}, where the Drinfeld realization is obtained. We consider its positive part $\dto$. Following \cite[4.1]{HSTY08}, fix $\hbar\in\mathbb{C}-\mathbb{Z}\pi\sqrt{-1}$, for any $u\in\mathbb{C}$ let 
 \begin{equation}
 v^{u}\coloneqq\text{exp}(u\hbar)=\sum_{n=0}^{\infty}\frac{(u\hbar)^{n}}{n!},\ \ \ v\coloneqq v^{1}.
 \end{equation}
 We also assume that $v$ is generic, that is $v^{ku}\neq 1$ for all $u\in\{1,\theta,\theta+1\}$ and $k\in\mathbb{N}$. $\dto$ is the $\mathbb{C}$-superalgebra with generators $\{e_{i,k}\}_{1\leqslant i\leqslant 3}^{k\in\mathbb{Z}}$, in which the parities are $p(e_{i,k})=1$ for any $i=1,2,3$ and $k\in\mathbb{Z}$, and  the following relations
\begin{equation}
\begin{aligned}
&[e_{i,k},e_{i,l}]=0,\ \ \ k,l\in \mathbb{Z},1\leqslant i\leqslant 3\\
&e_{i,k+1}e_{j,l}+v^{a_{ij}}e_{j,l}e_{i,k+1}=v^{a_{ij}}e_{i,k}e_{j,l+1}+e_{j,l+1}e_{i,k},\ \ \ a_{ij}\neq 0,k,l\in\mathbb{Z}\\
&[\theta]_{v}[[e_{1,r},e_{2,k}]_{v^{-1}},e_{3,l}]_{v}=[[e_{1,r},e_{3,l}]_{v^{-\theta}},e_{2,k}]_{v^{\theta}},\ \ \ r,k,l\in\mathbb{Z}
\label{defrel}
\end{aligned}
\end{equation}
where $[u]_{v}:=\frac{v^{u}-v^{-u}}{v-v^{-1}}$ for $u\in \mathbb{C}$.

We define the quantum affine root vectors by $E_{\alpha_{i}}(k)=e_{i,k}$, $E_{\alpha_{ij}}(k)=[e_{i,k},e_{j,0}]_{v^{-a_{ij}}}$, $E_{\alpha_{123}}(k)=[[e_{1,k},e_{2,0}]_{v^{-1}},e_{3,0}]_{v}$ for any $k\in\mathbb{Z}$. Let $H$ be the set of functions $h\colon \Psi^{+}\times \mathbb{Z}\rightarrow \mathbb{N}$ with finite support and such that $h(\beta,k)\leqslant 1$ if $\beta\in\Psi_{1}^{+}$. For each $h\in H$ we have the ordered monomial $E_{h}:=\prod_{(\beta,k)\in \Psi^{+}\times\mathbb{Z}}E_{\beta}(k)^{h(\beta,k)}$. Let $U'\subset\dto$ be the spanning set of these $E_{h}$ over $\mathbb{C}$.

\begin{Thm}
The set of ordered monomials $E_{h}$ is a spanning set for $\dto$.
\end{Thm}
 
\begin{proof}
Same to the proof of Proposition \ref{commutation}, by repeatedly using Lemma \ref{bra}, we get the commutation relations between these quantum affine root vectors. Specifically, for any $\beta,\beta'\in \Psi^{+}$, $k,l\in\mathbb{Z}$, we have $[E_{\beta}(k),E_{\beta'}(l)]_{v^{-\beta\cdot\beta'}}\in U'$. For $\beta=\alpha_{i},\beta'=\alpha_{j}$ or $\beta=\alpha_{i},\beta'=\alpha_{ij}$ or $\beta=\beta'=\alpha_{ij}$, it is the same as Proposition \ref{commutation}. For $\beta=\alpha_{ij},\beta'=\alpha_{k}$, $k\neq i,j$, we have $[[e_{1,r},e_{2,0}]_{v^{-1}},e_{3,k}]_{v}\approx[[e_{1,r},e_{3,k}]_{v^{-\theta}},e_{2,0}]_{v^{\theta}}\approx E_{\alpha_{123}}(r+k)+\sum E_{\alpha_{12}}(l)E_{\alpha_{3}}(l')+\sum E_{\alpha_{1}}(l)E_{\alpha_{23}}(l')\in U'$. For $\beta=\alpha_{ijk},\beta'=\alpha_{i}$, we have $[[[e_{1,r},e_{2,0}]_{v^{-1}},e_{3,0}]_{v},e_{2,k}]_{v^{\theta}}\approx [[E_{\alpha_{13}(r-k)},e_{2,k}]_{v^{\theta}},e_{2,k}]_{v^{\theta}}+\sum [[e_{1,l}e_{2,l'},e_{3,0}]_{v},e_{2,k}]_{v^{\theta}}\in U'$. The remaining cases are similar.

Now using these commutation relations, any product of quantum affine root vectors $E_{\beta}(k)E_{\beta'}(l)$ can be written as a linear combination of $E_{h}$ of the same length. Same to the proof of Theorem \ref{pbw2}, by induction on the length of elements, we show any element of $\dto$ is a finite linear combination of $E_{h}$. We omit the details.
\end{proof}

\subsection{Shuffle algebra $\Omega$}
Consider $\Omega=\bigoplus_{\underline{k}=(k_{1},k_{2},k_{3})\in \mathbb{N}^{3}}\Omega_{\underline{k}}$, where $\Omega_{\underline{k}}$ consists of rational functions $F$ in the variables $\{x_{i,r}\}_{1\leqslant i\leqslant 3}^{1\leqslant r\leqslant k_{i}}$ which satisfies:
\begin{enumerate}[leftmargin=*]
\item $F$ is skew-symmetric with respect to $\{x_{i,r}\}_{1\leqslant r\leqslant k_{i}}$ for any $1\leqslant i\leqslant 3$.
\item $F=\frac{f}{\prod_{1\leqslant i<j\leqslant 3,1\leqslant r\leqslant k_{i},1\leqslant s\leqslant k_{j}}(x_{i,r}-x_{j,s})}$, where $f\in \mathbb{C}[x_{i,r}^{\pm 1}]_{1\leqslant i\leqslant 3}^{1\leqslant r\leqslant k_{i}}$ is a Laurent polynomial.
\item $F$ satisfies the wheel condition, that is $F(\{x_{i,r}\}_{1\leqslant i\leqslant 3}^{1\leqslant r\leqslant k_{i}})=0$ once $x_{1,r}=v^{-1}x_{2,s}=v^{\theta}x_{3,w}$ or $x_{1,r}=vx_{2,s}=v^{-\theta}x_{3,w}$ for some $1\leqslant r\leqslant k_{1},1\leqslant s\leqslant k_{2},1\leqslant w\leqslant k_{3}$. 
 \end{enumerate}
 
We also fix an $3\times 3$ matrix of rational functions $(\omega_{i,j}(z))_{1\leqslant i,j\leqslant 3}\in\text{Mat}_{3\times 3}(\mathbb{C}(z))$ by setting
\begin{equation}
\begin{aligned}
&\omega_{i,j}(z)=-\omega_{j,i}(z)=\frac{z-v^{-a_{ij}}}{z-1},\ \ 1\leqslant i<j\leqslant 3\\
&\omega_{i,i}(z)=1,\ \ 1\leqslant i\leqslant 3.
\end{aligned}
\end{equation}
 
Denote by $\mathfrak{S}_{\underline{k}}=\mathfrak{S}_{k_{1}}\times\mathfrak{S}_{k_{2}}\times\mathfrak{S}_{k_{3}}$. For any $F\in \Omega_{\underline{k}},G\in \Omega_{\underline{l}}$, define their shuffle product $F\star G\in \Omega_{\underline{k}+\underline{l}}$ by $F\star G=$
\begin{equation}
\begin{aligned}
\text{ASym}_{\mathfrak{S}_{\underline{k}+\underline{l}}}\Big(F(\{x_{i,r}\}_{1\leqslant i\leqslant 3}^{1\leqslant r\leqslant k_{i}})\cdot G(\{x_{j,s}\}_{1\leqslant j\leqslant 3}^{k_{j}<s\leqslant k_{j}+l_{j}})\prod_{1\leqslant i,j\leqslant 3}\prod_{r\leqslant k_{i}}^{s>k_{j}}\omega_{i,j}(\frac{x_{i,r}}{x_{j,s}})\Big).\label{shuffleproduct}
\end{aligned}
\end{equation}
We know $\Omega$ is $\star$-closed, and $\Omega$ becomes an associative $\mathbb{C}$-algebra under $\star$. 

For an ordered monomial $E_h$, define its degree
$\deg(E_h)=\deg(h)=\underline{d}\in \mathbb{N}^{7}$ as a collection of
$d_{\beta}:=\sum_{r\in \mathbb{Z}} h(\beta,r)\in \mathbb{N}\ (\beta\in \Psi^+)$ ordered
with respect to the ordering on $\Psi^{+}$. We consider the lexicographical ordering on $\mathbb{N}^{7}$:
\begin{equation*}
  \{d_\beta\}_{\beta\in \Psi^+}>\{d'_\beta\}_{\beta\in \Psi^+}
  \ \mathrm{iff\ there\ is}\ \gamma\in \Psi^+\
  \mathrm{such\ that}\ d_\gamma>d'_\gamma\ \mathrm{and}\
  d_\beta=d'_\beta\ \mathrm{for\ all}\ \beta\prec\gamma.
\end{equation*}
Identifying simple roots as a basis for $\mathbb{N}^{3}$, for any $\underline{d}\in \mathbb{N}^{7}$ we define its grading $\text{gr}(\underline{d})=\sum_{\beta\in\Psi^{+}}d_{\beta}\beta\in\mathbb{N}^{3}$. Let us now define for any degree $\unl{d}$ a specialization map
\begin{equation}
\phi_{\unl{d}}\colon \Omega_{\text{gr}(\unl{d})}\rightarrow \mathbb{C}[\{w_{\beta,s}^{\pm 1}\}_{\beta\in\Psi^{+}}^{1\leqslant s\leqslant d_{\beta}}].
\end{equation}
Denote $\text{gr}(\unl{d})=\unl{k}$. For $1\leqslant i\leqslant 3$, we split the variables $\{x_{i,r}\}_{1\leqslant r\leqslant k_{i}}$ into groups $\{x^{\beta}_{i,s}\}_{\beta\in\Psi^{+}}^{1\leqslant s\leqslant d_{\beta}}$ corresponding to each $\beta\in\Psi^{+}$ and $1\leqslant s\leqslant d_{\beta}$. Now For any $F\in\Omega_{\unl{k}}$, let $f$ be the numerator part of $F$, define $\phi_{\underline{d}}(F)$ as the corresponding Laurent polynomial by specializing the variables in $f$ as follows:
\begin{equation}
\begin{aligned}
&x^{\beta}_{1,s}\mapsto w_{\beta,s},\ x^{\beta}_{2,s}\mapsto vw_{\beta,s},\ \beta\in\Psi^{+},\\
&x^{\beta}_{3,s}\mapsto v^{\theta}w_{\beta,s},\ \beta=\alpha_{13},\alpha_{123};\ x^{\beta}_{3,s}\mapsto v^{-\theta}w_{\beta,s},\ \beta=\alpha_{23},\alpha_{3}.
\end{aligned}
\end{equation}
Since $F\in \Omega_{\unl{k}}$ is skew-symmetric with respect to $\mathfrak{S}_{\unl{k}}$, different choices of our splitting of the variables only occur different signs in the specialization $\phi_{\unl{d}}(F)$ and we can ignore them.

\begin{Thm}\label{d21}
$e_{i,k}\mapsto x_{i}^{k}$ induces a $\mathbb{C}$-algebra isomorphism $\varphi\colon\dto\xrightarrow{\sim}\Omega$.
\end{Thm}

\begin{proof}
It is straightforward to check that the assignment $e_{i,k}\mapsto x_{i}^{k}$ induces an algebra morphism $\varphi$ from $\dto$ to $\Omega$. For example let us verify the last defining relation in \eqref{defrel}. Under $\varphi$, the left side of this equation is equal to
\begin{equation}
\begin{aligned}
&[\theta]_{v}[[x_{1,1}^{r},x_{2,1}^{k}]_{v^{-1}},x_{3,1}^{l}]_{v}\\&=[\theta]_{v}[x_{1,1}^{r}\star x_{2,1}^{k}+v^{-1}x_{2,1}^{k}\star x_{1,1}^{r},x_{3,1}^{l}]_{v}\\
&=\frac{[\theta]_{v}(1-v^{-2})}{x_{1,1}-x_{2,1}}[x_{1,1}^{r+1}x_{2,1}^{k},x_{3,1}^{l}]_{v}\\
&=\frac{[\theta]_{v}(1-v^{-2})}{x_{1,1}-x_{2,1}}(x_{1,1}^{r+1}x_{2,1}^{k}\star x_{3,1}^{l}-vx_{3,1}^{l}\star x_{1,1}^{r+1}x_{2,1}^{k})\\
&=Z\cdot[(x_{1,1}-v^{-\theta}x_{3,1})(x_{2,1}-v^{1+\theta}x_{3,1})-v(x_{3,1}-v^{-\theta}x_{1,1})(x_{3,1}-v^{1+\theta}x_{2,1})]\\
&=Z\cdot[(1-v^{2})x_{1,1}x_{2,1}+(v^{\theta+2}-v^{-\theta})x_{2,1}x_{3,1}+(v^{-\theta+1}-v^{\theta+1})x_{1,1}x_{3,1})],
\end{aligned}
\end{equation}
where $Z=\frac{[\theta]_{v}(1-v^{-2})x_{1,1}^{r+1}x_{2,1}^{k}x_{3,1}^{l}}{(x_{1,1}-x_{2,1})(x_{2,1}-x_{3,1})(x_{1,1}-x_{3,1})}$. On the right side similar computations give the same result.
In particular, we have $\varphi(E_{\alpha_{ij}}(k))=\frac{(1-v^{-2a_{ij}})x_{i,1}^{k+1}}{x_{i,1}-x_{j,1}}$ for $1\leqslant i<j\leqslant 3$ and $\varphi(E_{\alpha_{123}}(k))=$
$$\frac{(1-v^{-2})x_{1,1}^{k+1}[(1-v^{2})x_{1,1}x_{2,1}+(v^{\theta+2}-v^{-\theta})x_{2,1}x_{3,1}+(v^{-\theta+1}-v^{\theta+1})x_{1,1}x_{3,1}]}{(x_{1,1}-x_{2,1})(x_{2,1}-x_{3,1})(x_{1,1}-x_{3,1})}.$$

For $\deg(h)=\underline{d}$ we have $\phi_{\underline{d}}(\varphi(E_{h}))=c\cdot\prod_{\beta\prec\beta'}G_{\beta,\beta'}\prod_{\beta\in\Psi^{+}}G_{\beta}$ where $c$ is some non-zero constant and 
\begin{equation}
G_{\beta}=\left\{ \begin{aligned}
&w_{\beta,1}^{r_{\beta,1}}\star\cdots\star w_{\beta,d_{\beta}}^{r_{\beta,d_{\beta}}},\ \ \beta=\alpha_{1},\alpha_{2},\alpha_{3},\\
&\prod_{1\leqslant s<r\leqslant d_{\beta}}(w_{\beta,s}-w_{\beta,r})^{2}w_{\beta,1}^{r_{\beta,1}}\star\cdots\star w_{\beta,d_{\beta}}^{r_{\beta,d_{\beta}}},\ \ \beta=\alpha_{12},\alpha_{13},\alpha_{23},\\
&\prod_{1\leqslant s\neq r\leqslant d_{\beta}}(w_{\beta,s}-w_{\beta,r})(w_{\beta,s}-v^{-2}w_{\beta,r})(w_{\beta,s}-v^{-2\theta}w_{\beta,r})\\
&\cdot w_{\beta,1}^{r_{\beta,1}+2}\star\cdots\star w_{\beta,d_{\beta}}^{r_{\beta,d_{\beta}}+2},\ \ \beta=\alpha_{123},
\end{aligned}
\right.
\end{equation}
\begin{equation}
G_{\beta\prec\beta'}=\prod_{1\leqslant s\leqslant d_{\beta}}^{1\leqslant r\leqslant d_{\beta'}}\prod_{i\in[\beta],j\in[\beta']}y^{\beta,\beta'}_{i,j}(w_{\beta,s},w_{\beta',r}),
\end{equation}
where $\{r_{\beta,1},\dots,r_{\beta,d_{\beta}}\}$ is the support of $h$ restricted on $\beta$ and the shuffle element $w_{\beta,1}^{r_{\beta,1}}\star\cdots\star w_{\beta,d_{\beta}}^{r_{\beta,d_{\beta}}}$ is defined as monomial basis of skew-symmetric Laurent polynomials for odd root $\beta$ or as Hall-Littlewood basis of symmetric Laurent polynomials for even root $\beta$. And the function $y^{\beta,\beta'}_{i,j}(a,b)$ is defined as $y^{\beta,\beta'}_{1,2}=a-b,y^{\beta,\beta'}_{2,1}=a-v^{-2}b$ for any $\beta\prec\beta'$; $y^{
\beta,\beta'}_{1,3}=a-b,y^{\beta,\beta'}_{2,3}=a-v^{2\theta}b$ if $\beta'=\alpha_{13},\alpha_{123}$; $y^{
\beta,\beta'}_{1,3}=a-v^{-2\theta}b,y^{\beta,\beta'}_{2,3}=a-b$ if $\beta'=\alpha_{3},\alpha_{23}$; $y^{\beta,\beta'}_{3,1}=a-v^{-2\theta}b,y^{\beta,\beta'}_{3,2}=a-v^{2}b$ if $\beta=\alpha_{13},\alpha_{123}$; $y^{\beta,\beta'}_{i,i}=1$ for any $\beta\prec\beta'$ and $1\leqslant i\leqslant 3$. Same to \cite{Tsy18} we have to prove that $\phi_{\underline{d}'}(\varphi(E_{h}))=0$ for any $\underline{d}'<\deg(h)$. Recall that each term of $\varphi(E_{h})$ is corresponding to some permutation $\sigma\times\tau\times\mu$, and we will prove that each term is zero under specialization $\phi_{\underline{d}'}$.  Let $\text{deg}(h)=(d_{\beta}),\text{deg}(h')=(d'_{\beta})$, then there are the following cases.
\begin{itemize}[leftmargin=*]
\item $d'_{\alpha_{1}}<d_{\alpha_{1}}$, then if for some $1\leqslant i\leqslant d_{\alpha_{1}}$ we have $\sigma(i)>d_{\alpha_{1}}$, then this term is zero under $\phi_{\underline{d}'}$ because the $x_{1,\sigma(i)}$ will be mapping to some $w_{\beta,s}$ and some $x_{2,r}$(or $x_{3,r}$) will be mapping to $vw_{\beta,s}$(or $v^{\theta}w_{\beta,s}$), and the term contains the factor $x_{1,\sigma(i)}-v^{-1}x_{2,r}$(or $x_{1,\sigma(i)}-v^{-\theta}x_{3,r}$); if for any $1\leqslant i\leqslant d_{\alpha_{1}}$ we have $1\leqslant \sigma(i)\leqslant d_{\alpha_{1}}$, then since $d'_{\alpha_{1}}<d_{\alpha_{1}}$ under  $\phi_{\underline{d}'}$ there will be some $x_{1,i}$ for $1\leqslant i\leqslant d_{\alpha_{1}}$ and $x_{2,r}$(or $x_{3,r}$) that is mapping to some $w_{\beta,s}$ and $vw_{\beta,s}$(or $v^{\theta}w_{\beta,s}$), and for any such $\sigma$ the term contains the factor $x_{1,i}-v^{-1}x_{2,r}$(or $x_{1,i}-v^{-\theta}x_{3,r}$).

\item $d'_{\alpha_{1}}=d_{\alpha_{1}},d'_{\alpha_{13}}<d_{\alpha_{13}}$, then if there is some $1\leqslant i\leqslant d_{\alpha_{1}}$ such that $\sigma(1)>d_{\alpha_{1}}$, same for the arguments in the last case, the term is zero under $\phi_{\underline{d}'}$; otherwize for any $d_{\alpha_{1}}+1
\leqslant i\leqslant d_{\alpha_{1}}+d_{\alpha_{13}}$ we have $\sigma(i)>d_{\alpha_{1}}$, then there will be some $x_{1,s}$ and $x_{2,r}$ corresponding to each term that will be mapping to some $w_{\beta,t}$ and $vw_{\beta,t}$ and the term contains the factor $x_{1,s}-v^{-1}x_{2,r}$.

\item $d'_{\alpha_{1}}=d_{\alpha_{1}},d'_{\alpha_{13}}=d_{\alpha_{13}},d'_{\alpha_{12}}<d_{\alpha_{12}}$, then if for some $d_{\alpha_{1}}+d_{\alpha_{13}}+1\leqslant i\leqslant d_{\alpha_{1}}+d_{\alpha_{13}}+d_{\alpha_{12}}$ we have $\sigma(i)\leqslant d_{\alpha_{1}}+d_{\alpha_{13}}$, same for the arguments in the above cases we have the term is zero under $\phi_{\underline{d}'}$; otherwise there will be some $x_{1,s}$ and $x_{3,r}$ corresponding to each term that will be mapping to some $w_{\beta,t}$ and $v^{\theta}w_{\beta,t}$ and the term contains the factor $x_{1,s}-v^{-\theta}x_{3,r}$.

\item $d'_{\alpha_{2}}<d_{\alpha_{2}}$ and $d'_{\beta}=d_{\beta}$ for any $\beta\prec\alpha_{2}$, then if there is some $\beta=\alpha_{1},\alpha_{13},\alpha_{12},\alpha_{123}$ and such that $\sigma(x^{\beta}_{1,s})=x^{\beta'}_{1,r}$ for some $\beta'\neq\beta$, same for arguments in the above cases the term is zero under $\phi_{\underline{d}'}$; otherwise there will be some $x_{2,s}$ and $x_{3,r}$ corresponding to each term that will be mapping to some $vw_{\alpha_{23},t}$ and $v^{-\theta}w_{\alpha_{23},t}$ and the term contains the factor $x_{2,s}-v^{\theta+1}x_{3,r}$.
\end{itemize}
Hence we get $\{E_{h}\}_{h\in H}$ are PBW type bases for $\dto$ and $\varphi$ is injective.

For surjectivity of $\varphi$, by Remark \ref{remark} we only need to prove that given $h\in H$ such that $\text{gr}(h)=\underline{k}$ and $\text{deg}(h)=(d_{\beta})_{\beta\in \Psi^{+}}$, if for any $\text{gr}(h')=\text{gr}(h)$ and $\text{deg}(h')<\text{deg}(h)$ we have $\phi_{\underline{d}'}(F)=0$, then $\phi_{\underline{d}}(F)$ is a linear combination of some $\phi_{\underline{d}}(\varphi(E_{h}))$ for any $F\in \Omega_{\underline{k}}$. Actually, we only need to consider the case where there are only two positive roots $\beta\prec\beta'$ such that $d_{\beta},d_{\beta'}\neq 0$, and this can be done by case by case study. We give details of proof for some cases, other cases are similar.
\begin{itemize}[leftmargin=*]
\item For cases such as $(\beta,\beta')=(\alpha_{i},\alpha_{j}),(\alpha_{i},\alpha_{ij}),(\alpha_{ij},\alpha_{j})$, where $1\leqslant i<j\leqslant 3$, it is the same as Remark \ref{remark}.

\item For $(\beta,\beta')=(\alpha_{1},\alpha_{23}),(\alpha_{13},\alpha_{2}),(\alpha_{12},\alpha_{3})$, we consider the case $(\beta,\beta')=(\alpha_{13},\alpha_{2})$. We have $\phi_{\underline{d}}(\varphi(E_{h}))=\prod_{1\leqslant s<r\leqslant d_{\beta}}(w_{\beta,s}-w_{\beta,r})^{2}\prod_{1\leqslant s<r\leqslant d_{\beta'}}(w_{\beta',s}-w_{\beta',r})\prod_{1\leqslant s,\leqslant d_{\beta}}^{1\leqslant r\leqslant d_{\beta'}}(w_{\beta,s}-w_{\beta',r})(w_{\beta,s}-v^{2}w_{\beta',r})\cdot f$, in which $f\in \mathbb{C}[w_{\beta,s},w_{\beta',r}]^{\mathfrak{S}_{d_{\beta}}\times\mathfrak{S}_{d_{\beta'}}}_{1\leqslant s\leqslant d_{\beta},1\leqslant r\leqslant d_{\beta'}}$. Now for any $F\in \Omega_{(d_{\beta},d_{\beta'},d_{\beta})}$, $F$ is skew-symmetric, hence $\phi_{\underline{d}}(F)$ has the factor $\prod_{1\leqslant s<r\leqslant d_{\beta}}(w_{\beta,s}-w_{\beta,r})^{2}\prod_{1\leqslant s<r\leqslant d_{\beta'}}(w_{\beta',s}-w_{\beta',r})$. Under specialization $\phi_{\underline{d}}$ the wheel condition becomes $\phi_{\underline{d}}(F)=0$ once $w_{\beta,s}=v^{2}w_{\beta',r}$, hence giving us the factor $\prod_{1\leqslant s\leqslant d_{\beta}}^{1\leqslant r\leqslant d_{\beta'}}(w_{\beta,s}-v^{2}w_{\beta',r})$. Finally, let $\text{deg}(h')=(d_{\beta}-1,d_{\alpha_{123}}=1,d_{\beta'}-1)$, then $\text{deg}(h')<\text{deg}(h)$, hence $\phi_{\underline{d}'}(F)=0$, and gives us the last factor $\prod_{1\leqslant s\leqslant d_{\beta}}^{1\leqslant r\leqslant d_{\beta'}}(w_{\beta,s}-w_{\beta',r})$. 

\item For $(\beta,\beta')=(\alpha_{13},\alpha_{12}),(\alpha_{12},\alpha_{23}),(\alpha_{13},\alpha_{23})$, We consider the case $\beta=\alpha_{12},\beta'=\alpha_{23}$. We have $\phi_{\underline{d}}(\varphi(E_{h}))=\prod_{1\leqslant s<r\leqslant d_{\beta}}(w_{\beta,s}-w_{\beta,r})^{2}\prod_{1\leqslant s<r\leqslant d_{\beta'}}(w_{\beta',s}-w_{\beta',r})^{2}\prod_{1\leqslant s,\leqslant d_{\beta}}^{1\leqslant r\leqslant d_{\beta'}}(w_{\beta,s}-w_{\beta',r})^{2}(w_{\beta,s}-v^{-2\theta}w_{\beta',r})\cdot f$. The skew-symmetrization gives the factor $\prod_{1\leqslant s<r\leqslant d_{\beta}}(w_{\beta,s}-w_{\beta,r})^{2}\prod_{1\leqslant s<r\leqslant d_{\beta'}}(w_{\beta',s}-w_{\beta',r})^{2}\prod_{1\leqslant s\leqslant d_{\beta}}^{1\leqslant r\leqslant d_{\beta'}}(w_{\beta,s}-w_{\beta',r})$. Under specialization $\phi_{\underline{d}}$ the wheel condition becomes $\phi_{\underline{d}}(F)=0$ once $w_{\beta,s}=w_{\beta',r}$, hence giving us the factor $\prod_{1\leqslant s\leqslant d_{\beta}}^{1\leqslant r\leqslant d_{\beta'}}(w_{\beta,s}-v^{2\theta}w_{\beta',r})$. Finally, let $\text{deg}(h')=(d_{\beta}-1,d_{\alpha_{123}}=1,d_{\alpha_{2}}=1,d_{\alpha_{23}}-1)$, then $\text{deg}(h')<\text{deg}(h)$, hence $\phi_{\underline{d}'}(F)=0$, and gives us the last factor $\prod_{1\leqslant s\leqslant d_{\beta}}^{1\leqslant r\leqslant d_{\beta'}}(w_{\beta,s}-v^{-2\theta}w_{\beta',r})$. 

\item For $(\beta,\beta')=(\alpha_{1},\alpha_{123}),(\alpha_{123},\alpha_{2}),(\alpha_{123},\alpha_{3})$, we consider the case $\beta=\alpha_{1},\beta'=\alpha_{123}$. We have $\phi_{\underline{d}}(\varphi(E_{h}))=\prod_{1\leqslant s<r\leqslant d_{\beta}}(w_{\beta,s}-w_{\beta,r})\prod_{1\leqslant s<r\leqslant d_{\beta'}}(w_{\beta',s}-w_{\beta',r})^{3}\prod_{s\neq r}(w_{\beta',s}-v^{-2}w_{\beta',r})(w_{\beta',s}-v^{-2\theta}w_{\beta',r})\prod_{1\leqslant s,\leqslant d_{\beta}}^{1\leqslant r\leqslant d_{\beta'}}(w_{\beta,s}-w_{\beta',r})^{2}\cdot f$. The skew-symmetrization gives the factor $\prod_{1\leqslant s<r\leqslant d_{\beta}}(w_{\beta,s}-w_{\beta,r})\prod_{1\leqslant s<r\leqslant d_{\beta'}}(w_{\beta',s}-w_{\beta',r})^{3}\prod_{1\leqslant s,\leqslant d_{\beta}}^{1\leqslant r\leqslant d_{\beta'}}(w_{\beta,s}-w_{\beta',r})$. The wheel condition becomes $\phi_{\underline{d}}(F)=0$ once $w_{\beta',s}=v^{-2}w_{\beta',r}$ or $w_{\beta',s}=v^{-2\theta}w_{\beta',r}$. Let $\text{deg}(h')=(d_{\beta}-1,d_{\alpha_{12}}=1,d_{\alpha_{13}}=1,d_{\alpha_{123}}-1)$, then $\phi_{\underline{d}'}(F)=0$ gives the factor $\prod_{1\leqslant s,\leqslant d_{\beta}}^{1\leqslant r\leqslant d_{\beta'}}(w_{\beta,s}-w_{\beta',r})$.

\item For $(\beta,\beta')=(\alpha_{12},\alpha_{123}),(\alpha_{13},\alpha_{123}),(\alpha_{123},\alpha_{23})$, we consider the case $\beta=\alpha_{12},\beta'=\alpha_{123}$. We have $\phi_{\underline{d}}(\varphi(E_{h}))=\prod_{1\leqslant s<r\leqslant d_{\beta}}(w_{\beta,s}-w_{\beta,r})^{2}\prod_{1\leqslant s<r\leqslant d_{\beta'}}(w_{\beta',s}-w_{\beta',r})^{3}\prod_{1\leqslant s\neq r\leqslant d_{\beta'}}(w_{\beta',s}-v^{-2}w_{\beta',r})(w_{\beta',s}-v^{-2\theta}w_{\beta',r})\prod_{1\leqslant s,\leqslant d_{\beta}}^{1\leqslant r\leqslant d_{\beta'}}(w_{\beta,s}-w_{\beta',r})^{2}(w_{\beta,s}-v^{-2}w_{\beta',r})(w_{\beta,s}-v^{2\theta}w_{\beta,r})\cdot f$. The wheel condition becomes $\phi_{\underline{d}}(F)=0$ once $w_{\beta',s}=v^{-2}w_{\beta',r}$ or $w_{\beta',s}=v^{-2\theta}w_{\beta',r}$ or $w_{\beta,s}=v^{-2}w_{\beta',r}$ or $w_{\beta,s}=v^{2\theta}w_{
\beta',r}$, hence giving us the factor $\prod_{1\leqslant s\neq r\leqslant d_{\beta'}}(w_{\beta',s}-v^{-2}w_{\beta',r})(w_{\beta',s}-v^{-2\theta}w_{\beta',r})\prod_{1\leqslant s,\leqslant d_{\beta}}^{1\leqslant r\leqslant d_{\beta'}}(w_{\beta,s}-v^{-2}w_{\beta',r})(w_{\beta,s}-v^{2\theta}w_{\beta,r})$. The remaining factors come from the skew-symmetrization.
\end{itemize}

This completes our proof.
\end{proof}

\subsection{Generalization to all Dynkin diagrams associated to $\mathfrak{D}(2,1;\theta)$}\footnote{The results in this subsection have been previously worked out by Tsymbaliuk (private communication).} In this subsection, we give shuffle algebra realization of quantum affine algebras corresponding to all Dynkin diagrams associated to $\mathfrak{D}(2,1;\theta)$, making the picture for this exceptional Lie superalgebra complete.

Besides the simple root system with complete fermionic roots, there are three other simple root systems associated to $\mathfrak{D}(2,1;\theta)$, which all contains one fermionic root and two bosonic roots. The only difference in these three cases is the position of fermionic root, hence we only need to consider the case corresponding to the following Cartan matrix
$$A=\begin{pmatrix}
2&-1&0\\
-1&0&-\theta\\
0&-1&2
\end{pmatrix},$$
where $\theta\neq 0,-1$. We denote the corresponding Lie superalgebra by $\mathfrak{D}_{2}(2,1;\theta)$. Let $d_{1}=d_{2}=1,d_{3}=\theta$, so that $(d_{i}a_{ij})_{1\leqslant i,j\leqslant 3}$ is symmetric. The positive roots are $\Psi^{+}=\{\alpha_{1}\prec\alpha_{1}+\alpha_{2}\prec\alpha_{1}+\alpha_{2}+\alpha_{3}\prec\alpha_{1}+2\alpha_{2}+\alpha_{3}\prec\alpha_{2}\prec\alpha_{2}+\alpha_{3}\prec\alpha_{3}\}$ with a fixed ordering. We denote the highest positive root by $\gamma$ and denote the other positive roots by $\alpha_{ij}$ as before. Still we assume that $v\in\mathbb{C}$ is generic, that is $v^{ku}\neq 1$ for all $u\in\{1,\theta,\theta+1\}$ and $k\in\mathbb{N}$. The positive part of quantum affine superalgebra $U_{v}^{>}(\widehat{\mathfrak{D}}_{2}(2,1;\theta))$ is the $\mathbb{C}$-superalgebra with generators $\{e_{i,k}\}_{1\leqslant i\leqslant 3}^{k\in\mathbb{Z}}$, in which the parities are $p(e_{i,k})=\overline{i-1}$ for any $k\in\mathbb{N}$, and  the following relations:
\begin{equation}
\begin{aligned}
&[e_{i,k},e_{j,l}]=0,\ \ \ a_{ij}=0,k,l\in\mathbb{Z}\\
&[e_{i,k},e_{j,l+1}]_{v^{-d_{i}a_{ij}}}=-[e_{j,l},e_{i,k+1}]_{v^{-d_{j}a_{ji}}},\ \ \ a_{ij}\neq 0,k,l\in\mathbb{Z}\\
&\text{Sym}_{k,l}[e_{i,k},[e_{i,l},e_{2,s}]_{v^{-d_{i}a_{i2}}}]_{v^{-d_{i}a_{i2}-2d_{i}}}=0,\ \ \ i=1,3,k,l,s\in\mathbb{Z}
\end{aligned}
\end{equation}
The quantum affine root vectors $E_{\beta}(k)$ and the ordered monomials $E_{h}$ are also defined similarly as before. Especially, we have $E_{\gamma}(k)=[E_{\alpha_{13}}(k),E_{\alpha_{2}}(0)]_{v^{1+\theta}}$. Standard arguments show that these ordered monomials span the whole positive part. Note that the difference between this case and the case for type $A(2|2)$ with distinguished simple root system is that there is no commutation relations between quantum affine root vectors $E_{\alpha_{13}}$ and $E_{\alpha_{2}}$, and there is one more quantum affine root vector $E_{\gamma}$ in the ordered monomials $E_{h}$.

Consider $\Omega'=\bigoplus_{\underline{k}=(k_{1},k_{2},k_{3})\in \mathbb{N}^{3}}\Omega'_{\underline{k}}$, where $\Omega'_{\underline{k}}$ consists of rational functions $F$ in the variables $\{x_{i,r}\}_{1\leqslant i\leqslant 3}^{1\leqslant r\leqslant k_{i}}$ which satisfies:
\begin{enumerate}[leftmargin=*]
\item $F$ is symmetric with respect to $\{x_{i,r}\}_{1\leqslant r\leqslant k_{i}}$ for $i=1,3$ and skew-symmetric with respect to $\{x_{2,r}\}_{1\leqslant r\leqslant k_{2}}$.
\item $F=\frac{f}{\prod_{1\leqslant i\leqslant 2,1\leqslant r\leqslant k_{i},1\leqslant s\leqslant k_{i+1}}(x_{i,r}-x_{i+1,s})}$, where $f\in \mathbb{C}[x_{i,r}^{\pm 1}]_{1\leqslant i\leqslant 3}^{1\leqslant r\leqslant k_{i}}$ is a Laurent polynomial.
\item $F$ satisfies the wheel condition, that is $F(\{x_{i,r}\}_{1\leqslant i\leqslant 3}^{1\leqslant r\leqslant k_{i}})=0$ once $x_{1,r_{1}}=v^{2}x_{1,r_{2}}=vx_{2,s}$ or $x_{3,t_{1}}=v^{2\theta}x_{3,t_{2}}=v^{\theta}x_{2,s}$ for some $1\leqslant r_{1},r_{2}\leqslant k_{1},1\leqslant s\leqslant k_{2},1\leqslant t_{1},t_{2}\leqslant k_{3}$. 
 \end{enumerate}
Let $\omega_{ij}(z)=\frac{z-v^{-d_{i}a_{ij}}}{z-1}$, then $\Omega'$ becomes an associative algebra under the shuffle product similar to \eqref{shuffleproduct} except that we take symmetrization instead of skew-symmetrization with respect to $\{x_{1,r}\}$ and $\{x_{3,s}\}$. Now we have
\begin{Thm}\label{d212}
$e_{i,k}\mapsto x_{i}^{k}$ induces a $\mathbb{C}$-algebra isomorphism $\varphi\colon U_{v}^{>}(\widehat{\mathfrak{D}}_{2}(2,1;\theta))\xrightarrow{\sim}\Omega'$.
\end{Thm}
\begin{proof}
The only difficulty is that we need to define the specialization map corresponding to $\Omega'$. Now for any $E_{h}$, we label the variables in $\varphi(E_{h})$ by $\{x_{i,s}^{\beta}\}_{i\in[\beta],1\leqslant s\leqslant d_{\beta}}$ for $\beta\neq\gamma$ and by $\{x^{\beta}_{1,s},x^{\beta}_{2,1,s},x^{\beta}_{2,2,s},x^{\beta}_{3,s}\}_{1\leqslant s\leqslant d_{\beta}}$ for $\beta=\gamma$. Now define the specialization $\phi_{\underline{d}}(\varphi(E_{h}))\in\mathbb{C}[w_{\beta,s}^{\pm 1}]$ by specializing:
\begin{equation}
\begin{aligned}
&x^{\beta}_{1,s}\mapsto w_{\beta,s},\ x^{\beta}_{2,s}\mapsto v^{-1}w_{\beta,s},\ x^{\beta}_{3,s}\mapsto v^{-1-\theta}w_{\beta,s},\ \ \beta\neq\gamma\\
&x^{\gamma}_{1,s}\mapsto w_{\beta,s},\ x^{\gamma}_{2,1,s}\mapsto v^{-1}w_{\beta,s},\ x^{\gamma}_{2,2,s}\mapsto v^{-1-2\theta}w_{\beta,s},\ x^{\gamma}_{3,s}\mapsto v^{-1-\theta}w_{\beta,s}.
\end{aligned}
\end{equation}
Explicitly we have $\phi_{\underline{d}}(\varphi(E_{h}))=c\cdot\prod_{\beta\prec\beta'}G_{\beta,\beta'}\prod_{\beta\in\Psi^{+}}G_{\beta}\prod_{\beta\in\Psi^+}w_{\beta,1}^{r_{\beta,1}}\star\cdots\star w_{\beta,d_{\beta}}^{r_{\beta,d_{\beta}}}$ where $c$ is some non-zero constant and we have
\begin{itemize}[leftmargin=*]
\item $G_{\beta}=1,\beta=\alpha_{1},\alpha_{2},\alpha_{3}.$
\item $G_{\beta}=\prod_{1\leqslant s\neq r\leqslant d_{\beta}}(w_{\beta,s}-v^{2}w_{\beta,r}),\beta=\alpha_{12}.$
\item $G_{\beta}=\prod_{1\leqslant s\neq r\leqslant d_{\beta}}(w_{\beta,s}-v^{2\theta}w_{\beta,r}),\beta=\alpha_{23}.$
\item $G_{\beta}=\prod_{1\leqslant s\neq r\leqslant d_{\beta}}(w_{\beta,s}-v^{-2}w_{\beta,r})(w_{\beta,s}-v^{2\theta}w_{\beta,r}), \beta=\alpha_{13}.$
\item $G_{\beta}=\prod_{1\leqslant s\neq r\leqslant d_{\beta}}(w_{\beta,s}-w_{\beta,r})(w_{\beta,s}-v^{-2}w_{\beta,r})(w_{\beta,s}-v^{2\theta}w_{\beta,r}),\beta=\gamma.$
\item $G_{\beta,\beta'}=\prod_{1\leqslant s\leqslant d_{\beta}}^{1\leqslant r\leqslant d_{\beta'}}(w_{\beta,s}-w_{\beta',r}), (\beta,\beta')=(\alpha_{1},\alpha_{2}),(\alpha_{2},\alpha_{3}),(\alpha_{1},\alpha_{23}),(\alpha_{12},\alpha_{2}),(\alpha_{12},\alpha_{3})$,\\$(\alpha_{2},\alpha_{23}).$
\item $G_{\beta,\beta'}=\prod_{1\leqslant s\leqslant d_{\beta}}^{1\leqslant r\leqslant d_{\beta'}}(w_{\beta,s}-v^{-2}w_{\beta',r}),(\beta,\beta')=(\alpha_{1},\alpha_{12}),(\alpha_{1},\alpha_{13}).$
\item $G_{\beta,\beta'}=\prod_{1\leqslant s\leqslant d_{\beta}}^{1\leqslant r\leqslant d_{\beta'}}(w_{\beta,s}-v^{-2\theta}w_{\beta',r}),(\beta,\beta')=(\alpha_{13},\alpha_{3}),(\alpha_{23},\alpha_{3}).$
\item $G_{\beta,\beta'}=\prod_{1\leqslant s\leqslant d_{\beta}}^{1\leqslant r\leqslant d_{\beta'}}(w_{\beta,s}-v^{-2}w_{\beta',r})(w_{\beta,s}-v^{-2\theta}w_{\beta',r}),(\beta,\beta')=(\alpha_{1},\gamma).$
\item $G_{\beta,\beta'}=\prod_{1\leqslant s\leqslant d_{\beta}}^{1\leqslant r\leqslant d_{\beta'}}(w_{\beta,s}-w_{\beta',r})(w_{\beta,s}-v^{-2}w_{\beta',r})(w_{\beta,s}-v^{2}w_{\beta',r}),(\beta,\beta')=(\alpha_{12},\alpha_{13}).$
\item $G_{\beta,\beta'}=G_{\alpha_{12},\alpha_{13}}\cdot\prod_{1\leqslant s\leqslant d_{\beta}}^{1\leqslant r\leqslant d_{\beta'}}(w_{\beta,s}-v^{-2\theta}w_{\beta',r}),(\beta,\beta')=(\alpha_{12},\gamma).$
\item $G_{\beta,\beta'}=\prod_{1\leqslant s\leqslant d_{\beta}}^{1\leqslant r\leqslant d_{\beta'}}(w_{\beta,s}-w_{\beta',r})^{2},(\beta,\beta')=(\alpha_{12},\alpha_{23}).$
\item $G_{\beta,\beta'}=G_{\alpha_{12},\gamma}\cdot\prod_{1\leqslant s\leqslant d_{\beta}}^{1\leqslant r\leqslant d_{\beta'}}(w_{\beta,s}-v^{-2\theta}w_{\beta',r})(w_{\beta,s}-v^{2\theta}w_{\beta',r}),(\beta,\beta')=(\alpha_{13},\gamma).$
\item $G_{\beta,\beta'}=\prod_{1\leqslant s\leqslant d_{\beta}}^{1\leqslant r\leqslant d_{\beta'}}(w_{\beta,s}-w_{\beta',r})(w_{\beta,s}-v^{2\theta}w_{\beta',r}),(\beta,\beta')=(\alpha_{13},\alpha_{2}),(\gamma,\alpha_{2}).$
\item $G_{\beta,\beta'}=G_{\alpha_{13},\alpha_{2}}\cdot\prod_{1\leqslant s\leqslant d_{\beta}}^{1\leqslant r\leqslant d_{\beta'}}(w_{\beta,s}-v^{2\theta}w_{\beta',r}),(\beta,\beta')=(\alpha_{13},\alpha_{23}).$
\item $G_{\beta,\beta'}=G_{\gamma,\alpha_{2}}\cdot\prod_{1\leqslant s\leqslant d_{\beta}}^{1\leqslant r\leqslant d_{\beta'}}(w_{\beta,s}-v^{-2\theta}w_{\beta',r})(w_{\beta,s}-v^{2\theta}w_{\beta',r}),(\beta,\beta')=(\gamma,\alpha_{23}).$
\item $G_{\beta,\beta'}=\prod_{1\leqslant s\leqslant d_{\beta}}^{1\leqslant r\leqslant d_{\beta'}}(w_{\beta,s}-v^{-2\theta}w_{\beta',r})(w_{\beta,s}-v^{2\theta}w_{\beta',r}),(\beta,\beta')=(\gamma,\alpha_{3}).$
\end{itemize}
Now same to the proof of Theorem \ref{d21}, we have $\phi_{\unl{d}'}(\varphi(E_{h}))=0$ for any $\unl{d}'<\deg(h)$ and by looking at each pair of positive roots the wheel conditions give us the vanishing factors as above, thus completing our proof.
\end{proof}

\end{document}